\documentclass[12pt]{amsart}
\usepackage{amsmath,amssymb}
\usepackage{bm}
\usepackage{xcolor}
\usepackage{mathrsfs}
\usepackage{verbatim}
\usepackage{float}
\usepackage{setspace}
\usepackage{amsthm}
\usepackage{geometry}
\usepackage{color}
\usepackage{cite}
\newtheorem{theorem}{Theorem}[section]
\newtheorem{lemma}{Lemma}[section]
\newtheorem{remark}{Remark}[section]

  \makeatletter
  \@addtoreset{equation}{section}
  \makeatother

\begin{document}

\title[Discrete Gradgrad-Complexes in Three Dimensions]{Conforming Discrete Gradgrad-Complexes in Three Dimensions}

\author {Jun Hu}
\address{LMAM and School of Mathematical Sciences, Peking University,
  Beijing 100871, P. R. China. }
\email{  hujun@math.pku.edu.cn}

\author {Yizhou Liang}
\address{LMAM and School of Mathematical Sciences, Peking University,
  Beijing 100871, P. R. China. }
\email{ lyz2015@pku.edu.cn }

\thanks{The authors were supported by  NSFC
	projects 11625101 and 11421101.\\
The second author was supported by The Elite Program of Computational and Applied Mathematics for PhD Candidates in Peking University.}
\begin{abstract}
	In this paper,\ the first family of conforming discrete three dimensional Gradgrad-complexes consisting of finite element spaces is constructed.\ These discrete complexes are exact in the sense that the range of each discrete map is the kernel space of the succeeding one.\ These spaces can be used in the mixed form of the linearized Einstein-Bianchi system.\ 
\end{abstract}
\subjclass[2010]{65N30}
\maketitle
\section{Introduction}
Einstein's equations are a set of ten coupled and non-linear equations that relate to the Einstein tensor $G_{ab}$ and the energy-momentum tensor $T_{ab}$
\begin{equation*}
G_{ab} = \frac{8\pi G}{c^4}T_{ab},
\end{equation*}
where $G$ is the gravitational constant and $c$ is the speed of light.\ Because of the complexity,\ these equations can only be solved analytically in some special cases.\ A new approach to solve the Einstein's equations is based on the Einstein-Bianchi formulation\cite{MR1634577,MR1397127},\ using the full Bianchi identities and a decomposition of the Riemann tensor $R_{abcd}$
\begin{equation*}
R_{abcd} = M_{abcd} + W_{abcd},
\end{equation*}
with the so called Weyl tensor $W_{abcd}$ and another part $M_{abcd}$ depend on the Ricci tensor $R_{ab}$.\ In the case of vacuum,\ the Ricci tensor $R_{ab}$ is zero,\ which implies that $M_{abcd}$ is zero.\ A hyperbolic system for vacuum Einstein's equations is based on the Weyl tensor,\ using the Bel(see\cite{MR1634577,MR1397127}) electric and magnetic fields.
There,\ the so called linearized Einstein-Bianchi system from \cite{MR3450102}
\begin{equation*}
\begin{aligned}
&\dot{\mathbf{E}} + \operatorname{curl}\mathbf{B} = 0,\quad \operatorname{div}\mathbf{E}=0,\\
&\dot{\mathbf{B}} - \operatorname{curl}\mathbf{E} = 0,\quad \operatorname{div}\mathbf{B}=0,
\end{aligned}
\end{equation*}
is similar to Maxwell's equations but with symmetric and traceless tensor fields $\mathbf{E}$ and $\mathbf{B}$,\ respectively.\ Because of the essential difference of unknowns,\ classical numerical methods of electromagnetism don't work for this case.\ In \cite{MR3450102}, by introducing a new variable $\sigma(t) = \int_0^{t}\operatorname{div}\operatorname{div}\mathbf{E}\mathrm{d}s$,\ the linearized Einstein-Bianchi system can be realized as a Hodge wave equation
\begin{equation}\label{Intro 1}
\begin{aligned}
&\dot{\sigma} = \operatorname{div}\operatorname{div}\mathbf{E},\\
&\dot{\mathbf{E}} =-\operatorname{grad}\operatorname{grad}\sigma- \operatorname{sym}\operatorname{curl}\mathbf{B} ,\\
&\dot{\mathbf{B}} = \operatorname{curl}\mathbf{E} .
\end{aligned}
\end{equation}
Given initial conditions $\sigma(0),\mathbf{E}(0)$ and $\mathbf{B}(0)$,\ and with appropriate boundary conditions,\ the equation \eqref{Intro 1} is well-posed(see \cite{MR3450102}).

In the weak form of the new formulation of the linearized Einstein-Bianchi system(see \eqref{eq5.0},\eqref{eq5.1} below),\ the symmetry of the electric tensor field $\mathbf{E}$ is imposed strongly,\ namely the electric tensor field is sought in $C^0([0,T],H(\operatorname{curl},\Omega;\mathbb{S}))$,\ taking values in the space $\mathbb{S}:= \mathbb{R}^{3\times 3}_{\operatorname{sym}}$ of symmetric matrices.\ The construction of an appropriate finite element subspace of $H(\operatorname{curl},\Omega;\mathbb{S})$ using polynomial shape functions is a new challenging problem.\ As far as we know,\ there are no mixed finite elements for the Hodge wave equation \eqref{Intro 1} with symmetry of the electric part imposed strongly up until recently.\ In \cite{MR3450102},\ the mixed finite elements for the linearized Einstein-Bianchi system in which the symmetry is imposed weakly were proposed.

The mixed finite elements for the linearized Einstein-Bianchi system is closely related to the discretization of an associated differential complex.\ Such a Gradgrad-complex,\ introduced in \cite{MR4113080} to derive a Helmholtz-like decomposition for biharmonic problems in $\mathbb{R}^3$,\ is given by
\begin{equation}\label{Complex1}
P_1(\Omega)\stackrel{\subset}{\longrightarrow}H^2(\Omega;\mathbb{R}) \stackrel{\operatorname{gradgrad}}{\longrightarrow}H(\operatorname{curl},\Omega;\mathbb{S}) \stackrel{\operatorname{curl}}{\longrightarrow}H(\operatorname{div},\Omega;\mathbb{T})\stackrel{\operatorname{div}}{\longrightarrow}L^2(\Omega;\mathbb{R}^3)\stackrel{}{\longrightarrow}0,
\end{equation}
where the space $H(\operatorname{div},\Omega;\mathbb{T})$ consists of square-integrable tensors with square-integrable divergence,\ taking value in the space $\mathbb{T}$ of traceless matrices.\ The complex is exact provided that the domain $\Omega$ is contractible and Lipschitz\cite{MR4113080},\ that is,\ the range space of each map is the kernel space of the succeeding map.\ In this paper,\ the purpose is to construct conforming finite element spaces $U_{h} \subset H^2(\Omega;\mathbb{R}),\Sigma_{h}\subset H(\operatorname{curl},\Omega;\mathbb{S}),V_{h}\subset H(\operatorname{div},\Omega;\mathbb{T})$ and $Q_{h}\subset L^2(\Omega;\mathbb{R}^{3})$ such that 
\begin{equation}\label{Discrete complex}
P_1(\Omega)\stackrel{\subset}{\longrightarrow}U_{h} \stackrel{\operatorname{gradgrad}}{\longrightarrow}\Sigma_{h} \stackrel{\operatorname{curl}}{\longrightarrow}V_{h}\stackrel{\operatorname{div}}{\longrightarrow}Q_{h}\stackrel{}{\longrightarrow}0
\end{equation}
is an exact sub-complex of \eqref{Complex1}.

A natural point is to take $U_{h}$ to be the $H^2$-conforming finite element spaces introduced by \v{Z}en\'{\i}\v{s}ek\cite{MR350260} and Zhang\cite{MR2474112}.\ These spaces  consist of globally $C^1$ piecewise polynomials of degree 9 and higher,\ those are $C^4$ at vertices and $C^2$ on edges of the triangulation.\ Our constructions of finite elements of $\Sigma_{h}$ are motivated by the mixed finite elements for linear elasticity in\cite{MR3352360,MR3301063,MR3529252}.\ In these elements,\ the symmetric tensor is approximated by a so called $H(\operatorname{curl})$ bubble space enrichment of a $C^0$ space.\ The construction of $V_{h}$ is analogous.\ In order to prove the exactness property below,\ the finite elements $\Sigma_{h},V_{h}$ and associated bubble spaces have to be imposed enhanced regularity at sub-simplexes of the triangulation.\ Due to the complexity and high polynomial degree of the elements,\ their practical significance may be limited.\ However,\ this work is the first construction of conforming discrete Gradgrad-complexes consisting of finite element spaces in $\mathbb{R}^3$,\ and it provides insights to the development of simpler methods.

The rest of the paper is organized as follows.\ Section 2 introduces the notation.\ Section 3 defines the finite element spaces of $H(\operatorname{curl},\Omega;\mathbb{S})$.\ Section 4 states the definition of the finite element spaces of $H(\operatorname{div},\Omega;\mathbb{T})$,\ and proves a result of the divergence of the bubble space.\ Section 5 shows that the complex \eqref{Discrete complex} is exact.\ Section 6 uses the newly proposed elements to solve the mixed form of the linearized Einstein-Bianchi system and shows the error estimate.\ Some conclusions are given in Section 7.

\section{Notation}
Let $\Omega$ be a contractible domain with Lipschitz boundary $\partial \Omega$ of $\mathbb{R}^{3}$.\ Denote by $\mathbb{M}$ the space of\ $3\times 3$\ real matrices,\ and let $\mathbb{S}$ and $\mathbb{T}$ be the subspace of symmetric and traceless matrices,\ respectively.\ For $\sigma : \Omega \rightarrow \mathbb{M}$,\ the symmetric part of $\sigma$ is $\operatorname{sym}\sigma = (\sigma + \sigma^{T})/2$.

Let standard notation $H^{m}(D;X)$ denote the Sobolev space consisting of functions with domain $D\subset \mathbb{R}^{3}$,\ taking values in vector space $X$,\ and with all derivatives of order at most $m$ square-integrable.\ In the case $m=0$,\ set $H^0(D;X) = L^2(D;X)$.\ In this paper,\ $X$ will be either $\mathbb{M},\mathbb{S},\mathbb{T},\mathbb{R}$ or $\mathbb{R}^3$.\ If $X$ is clear in the context,\ $H^{m}(D;X)$ will be simplified as $H^{m}(D)$.\ Denote by $\|\cdot\|_{m,D}$ the norm of $H^{m}(D)$.\ Define
\begin{equation*}
\begin{aligned}
H(\operatorname{curl}, D ; \mathbb{S})=&\left\{\sigma \in L^{2}(D ; \mathbb{S}) \big| \operatorname{curl} \sigma\in L^{2}(D ; \mathbb{M})\right\},\\
H(\operatorname{div}, D ; \mathbb{T})=&\left\{\sigma \in L^{2}(D ; \mathbb{T}) \big| \operatorname{div} \sigma \in L^{2}(D ; \mathbb{R}^3)\right\},\\
\end{aligned}
\end{equation*} 
where the operators curl and div acting on a matrix field are obtained by applying the operators on each row.\ The norms of $H(\operatorname{curl})$ and $H(\operatorname{div})$ are defined by 
\begin{equation*}
\|\sigma\|_{H(\mathrm{curl},D)}^{2}=\|\sigma\|^{2}_{0,D}+\|\operatorname{curl} \sigma\|^{2}_{0,D},\quad \|\sigma\|_{H(\mathrm{div},D)}^{2}=\|\sigma\|^{2}_{0,D}+\|\operatorname{div}\sigma\|^{2}_{0,D}.
\end{equation*}

For $1 \leq p \leq \infty$,\ let $\mathcal{B}$ be a Banach space with norm $\|\cdot\|_{\mathcal{B}}$,\ and $T$ be a positive real number.\ Denote by $L^{p}([0,T],\mathcal{B})$ the space of functions $f:[0,T] \rightarrow \mathcal{B}$ with
\begin{equation*}
\|f\|_{L^{p}(\mathcal{B})}^{p} := \int_{0}^{T}\|f(t)\|_{\mathcal{B}}^{p} \mathrm{d} t < \infty \ (1\leq p < \infty),\quad \|f\|_{L^{\infty}(\mathcal{B})} :=\operatorname{ess} \sup _{t\in [0, T]}\|f(t)\|_{\mathcal{B}} < \infty.
\end{equation*}
Let $m$ be a positive integer.\ Similar to the Sobolev space,\ denote by $W^{m,p}([0,T],\mathcal{B})$ the space such that
\begin{equation*}
\|f\|_{W^{m, p}(\mathcal{B})}^{p} :=\sum_{l=0}^{m}\left\|\partial^{l} f / \partial t^{l}\right\|_{L^{p}(\mathcal{B})}^{p}<\infty.
\end{equation*} 
Similarly,\ let $C^{m}([0,T],\mathcal{B})$ denote the space of m-times continuously differentiable functions.

If $f\in\mathcal{X}\cap\mathcal{B}$,\ here $\mathcal{X}$ and $\mathcal{B}$ are two different Banach spaces,\ then define $\|f\|_{\mathcal{X}\cap\mathcal{B}} = \|f\|_{\mathcal{X}} + \|f\|_{\mathcal{B}}$.

For simplicity of presentation,\ let $\dot{f},\ddot{f},\dddot{f}$ denote $\partial f / \partial t,\partial^{2} f / \partial t^{2},\partial^{3} f / \partial t^{3}$ in the following sections.

Let $\{\mathcal{T}_{h}\}$ denote a family of regular tetrahedral grids on $\Omega$ with the grid size $h$.\ In the following,\ denote by $\mathscr{V}$ the set of vertices,\ $\mathscr{E}$ for edges and $\mathscr{F}$ for faces,\ the set of 3D cells by $\mathscr{T}$.\ Let $\mathcal{V},\mathcal{E},\mathcal{F}$ and $\mathcal{T}$ denote the number of vertices,\ edges,\ faces and tetrahedra in the triangulation,\ respectively.\ Let $P_{k}(D;X)$ denote the space of polynomials of degree $\leq k$ on a single simplex $D\in \mathcal{T}_{h}$,\ taking value in the space $X$.

\section{$H(\operatorname{curl})$ Finite Elements with Strong Symmetry}
This section considers $H(\operatorname{curl})$ conforming finite element spaces with strong symmetry.\ The idea is motivated by the mixed elements for linear elasticity introduced in\cite{MR3352360,MR3301063,MR3529252} where $H(\operatorname{div})$ bubble enriched spaces of Lagrange elements are constructed.\ To this end,\ first introduce a $H(\operatorname{curl})$ bubble function space.
\subsection{Full $\mathbf{H(\operatorname{curl})}$ bubble function space}
Let $\boldsymbol{x}_0,\boldsymbol{x}_1,\boldsymbol{x}_2,\boldsymbol{x}_3$ be the vertices of tetrahedron $K$.\ Denote by $\boldsymbol{f}_{i}$ the face of $K$ opposite to $\boldsymbol{x}_{i}$ ($0\leq i \leq 3$) with $\boldsymbol{n}_{i}$ the unit normal vector of the face $\boldsymbol{f}_{i}$ and $\lambda_{i}$ the i-th barycentric
coordinate.\ Hereafter,\ let $P_{k}^{(i)}(K;\mathbb{R})$ denote the space consisting of Lagrange nodal basis functions of degree $\leq k$ at all Lagrange nodal points on $\boldsymbol{f}_{i}$.

Define a $H(\operatorname{curl},K;\mathbb{S})$ bubble function space on $K$ for $k\geq 4$ as follows:
\begin{equation*}
\Sigma_{K,k,b} := \{\sum_{i=0}^3\lambda_{j}\lambda_{l}\lambda_{m}P_{k-3}^{(i)}(K;\mathbb{R})\boldsymbol{n}_{i}\boldsymbol{n}_{i}^{T}\} \oplus \{ \lambda_0\lambda_1\lambda_2\lambda_3P_{k-4}(K;\mathbb{S}) \},
\end{equation*}
where $\{i,j,l,m\}$ is a permutation of $\{0,1,2,3\}$.\ Note that the bubble function space has another equivalent form:
\begin{equation*}
\Sigma_{K,k,b} = \{\sum_{i=0}^{3}\lambda_{j}\lambda_{l}\lambda_{m}P_{k-3}(K;\mathbb{R})\boldsymbol{n}_{i}\boldsymbol{n}_{i}^{T}\}+ \{ \lambda_0\lambda_1\lambda_2\lambda_3P_{k-4}(K;\mathbb{S})\}.
\end{equation*}

Given a matrix $M = [m_1,m_2,m_3]^{T}$ and a vector $v$,\ define the vector product of matrix and vector by
\begin{equation*}
M\times v = [m_1\times v,m_2\times v,m_3\times v]^{T}.
\end{equation*}
Then the full $H(\operatorname{curl},K;\mathbb{S})$ bubble function space consisting of polynomials of degree $\leq k$ is defined by
\begin{equation*}
\Sigma_{\partial K, k, 0} :=\left\{\sigma \in P_{k}(K ; \mathbb{S}): \sigma \times \boldsymbol{n}\big|_{\partial K}=0\right\},
\end{equation*}
here $\boldsymbol{n}$ is the unit normal vector of $\partial K$.

The following theorem is crucial.
\begin{theorem}
	When $k\geq 4$,\ it holds that
	\begin{equation*}
	\Sigma_{K,k,b} =\Sigma_{\partial K, k, 0}.
	\end{equation*}
\end{theorem}
To prove this theorem,\ the following lemma is needed.

\begin{lemma}
	The matrices $\boldsymbol{n}_0\boldsymbol{n}_0^{T},\boldsymbol{n}_1\boldsymbol{n}_1^{T},\boldsymbol{n}_2\boldsymbol{n}_2^{T},\boldsymbol{n}_3\boldsymbol{n}_3^{T},(\boldsymbol{n}_1\boldsymbol{n}_2^{T} + \boldsymbol{n}_2\boldsymbol{n}_1^{T}),(\boldsymbol{n}_1\boldsymbol{n}_3^{T}+\boldsymbol{n}_3\boldsymbol{n}_1^{T})$ are linearly independent,\ and form a basis of $\mathbb{S}$.
\end{lemma}
\begin{proof}
	First,\ it is easy to show that $\boldsymbol{n}_1\boldsymbol{n}_1^{T},\boldsymbol{n}_2\boldsymbol{n}_2^{T},\boldsymbol{n}_3\boldsymbol{n}_3^{T},(\boldsymbol{n}_1\boldsymbol{n}_2^{T} + \boldsymbol{n}_2\boldsymbol{n}_1^{T}),(\boldsymbol{n}_1\boldsymbol{n}_3^{T}+\boldsymbol{n}_3\boldsymbol{n}_1^{T}),(\boldsymbol{n}_2\boldsymbol{n}_3^{T}+\boldsymbol{n}_3\boldsymbol{n}_2^{T})$ form a basis of $\mathbb{S}$.
	
	Observe that any three vectors of $\boldsymbol{n}_0,\boldsymbol{n}_1,\boldsymbol{n}_2,\boldsymbol{n}_3$ form a basis of $\mathbb{R}^3$,\ namely
	\begin{equation}\label{E1}
	\boldsymbol{n}_0 = a \boldsymbol{n}_1 + b\boldsymbol{n}_2 +c \boldsymbol{n}_3,
	\end{equation}
	where $a,b,c$ are nonzero constants.\ It leads to
	\begin{align*}
	\boldsymbol{n}_0\boldsymbol{n}_0^{T} = &a^2\boldsymbol{n}_1\boldsymbol{n}_1^{T} + b^2\boldsymbol{n}_2\boldsymbol{n}_2^{T} +c^2\boldsymbol{n}_3\boldsymbol{n}_3^{T}  + ab(\boldsymbol{n}_1\boldsymbol{n}_2^{T} + \boldsymbol{n}_2\boldsymbol{n}_1^{T})\\
	&\quad + ac(\boldsymbol{n}_1\boldsymbol{n}_3^{T} + \boldsymbol{n}_3\boldsymbol{n}_1^{T}) + bc(\boldsymbol{n}_2\boldsymbol{n}_3^{T} + \boldsymbol{n}_3\boldsymbol{n}_2^{T}),
	\end{align*}
	here the coefficient $bc$ is nonzero.\ This completes the proof.
\end{proof}

With the help of the lemma,\ we now turn to the proof of the above theorem.
\begin{proof}[Proof of Theorem 3.1]
	
	Consider a function $\tau \in \lambda_{j}\lambda_{l}\lambda_{m}P_{k-3}(K;\mathbb{R})\boldsymbol{n}_{i}\boldsymbol{n}_{i}^{T}$,\ which vanishes on $\boldsymbol{f}_{j},\boldsymbol{f}_{l},\boldsymbol{f}_{m}$.\ According to the definition of the vector product of matrices and vectors,\ it follows that $(\boldsymbol{n}_{i}\boldsymbol{n}_{i}^{T})\times \boldsymbol{n}_{i} = \boldsymbol{n}_{i}(\boldsymbol{n}_{i}\times \boldsymbol{n}_{i})^{T}$ vanishes,\ which implies that $\tau \times \boldsymbol{n}=0$ on $\partial K$.\ Since $\lambda_0\lambda_1\lambda_2\lambda_3P_{k-4}(K;\mathbb{S}) \subset \Sigma_{\partial K,k,0}$,\ it holds that
	\begin{equation*}
	\Sigma_{\partial K,k,0}  \supset \Sigma_{K,k,b}.
	\end{equation*}
	
	It remains to show the converse.\ Given $\tau \in \Sigma_{\partial K,k,0}$,\ it follows from Lemma 3.1 that
	\begin{equation*}
	\tau = \sum_{i=0}^{3}u_{i}\boldsymbol{n}_{i}\boldsymbol{n}_{i}^{T} + v(\boldsymbol{n}_1\boldsymbol{n}_2^{T} + \boldsymbol{n}_2\boldsymbol{n}_1^{T}) + w(\boldsymbol{n}_1\boldsymbol{n}_3^{T}+\boldsymbol{n}_3\boldsymbol{n}_1^{T}),
	\end{equation*}
	where the coefficient functions $u_0,u_1,u_2,u_3,v,w\in P_{k}(K;\mathbb{R})$.\ By the definition on face $\boldsymbol{f}_{i}$, it follows that $\tau\times\boldsymbol{n}_{i} = 0$.\ There are four cases for the proof.
	\begin{itemize}
		\item[1.] On face $\boldsymbol{f}_0$,\ the vector product of $\tau$ and $\boldsymbol{n}_0$ reads
		\begin{align*}
		\tau\times\boldsymbol{n}_0 = &\sum_{i= 1}^{3}u_{i}\boldsymbol{n}_{i}(\boldsymbol{n}_{i}\times\boldsymbol{n}_0)^{T}+ v(\boldsymbol{n}_1(\boldsymbol{n}_2\times\boldsymbol{n}_0)^{T} +\boldsymbol{n}_2(\boldsymbol{n}_1\times\boldsymbol{n}_0)^{T}) \\ &\quad+ w(\boldsymbol{n}_1(\boldsymbol{n}_3\times\boldsymbol{n}_0)^{T}+\boldsymbol{n}_3(\boldsymbol{n}_1\times\boldsymbol{n}_0)^{T}) = 0 .
		\end{align*}
		It can be reformulated as
		\begin{align*}
		\boldsymbol{n}_1(u_1(\boldsymbol{n}_1\times\boldsymbol{n}_0)^{T} + &v(\boldsymbol{n}_2\times\boldsymbol{n}_0)^{T} +w(\boldsymbol{n}_3\times\boldsymbol{n}_0)^{T}) +\boldsymbol{n}_2(u_2(\boldsymbol{n}_2\times\boldsymbol{n}_0)^{T} +\\ v(\boldsymbol{n}_1\times\boldsymbol{n}_0)^{T}) 
		&+\boldsymbol{n}_3(u_3(\boldsymbol{n}_3\times\boldsymbol{n}_0)^{T} + w(\boldsymbol{n}_1\times\boldsymbol{n}_0)^{T}) = 0\quad on \ \boldsymbol{f}_0.
		\end{align*}
		Because $\boldsymbol{n}_1,\boldsymbol{n}_2,\boldsymbol{n}_3$ are linearly independent,\ this implies that
		\begin{align*}
		&u_1(\boldsymbol{n}_1\times\boldsymbol{n}_0)^{T} + v(\boldsymbol{n}_2\times\boldsymbol{n}_0)^{T} +w(\boldsymbol{n}_3\times\boldsymbol{n}_0)^{T} = 0,\\
		&u_2(\boldsymbol{n}_2\times\boldsymbol{n}_0)^{T} + v(\boldsymbol{n}_1\times\boldsymbol{n}_0)^{T} = 0, \qquad \qquad \qquad\qquad  on \ \boldsymbol{f}_0\\
		&u_3(\boldsymbol{n}_3\times\boldsymbol{n}_0)^{T} + w(\boldsymbol{n}_1\times\boldsymbol{n}_0)^{T}= 0.
		\end{align*}
		According  to the fact that any two vectors of $\boldsymbol{n}_1\times\boldsymbol{n}_0,\boldsymbol{n}_2\times\boldsymbol{n}_0,\boldsymbol{n}_3\times\boldsymbol{n}_0$ are linearly independent,\ it follows that $u_1,u_2,u_3,v,w$ vanish on $\boldsymbol{f}_0$.\ Therefore $u_1,u_2,u_3,v,w$ contain a factor $\lambda_0$.
		
		\item[2.]On face $\boldsymbol{f}_1$,\ it follows from \eqref{E1} that
		\begin{align*}
		&\tau\times \boldsymbol{n}_1 = u_{0}\boldsymbol{n}_0((a\boldsymbol{n}_1+b\boldsymbol{n}_2+c\boldsymbol{n}_3)\times\boldsymbol{n}_1)^{T} + u_2\boldsymbol{n}_2(\boldsymbol{n}_2\times\boldsymbol{n}_1)^{T} +\\ &u_3\boldsymbol{n}_3(\boldsymbol{n}_3\times\boldsymbol{n}_1)^{T}
		 + v\boldsymbol{n}_{1}(\boldsymbol{n}_{2}\times\boldsymbol{n}_1)^{T} + w\boldsymbol{n}_{1}(\boldsymbol{n}_{3}\times\boldsymbol{n}_1)^{T} = 0\quad on \ \boldsymbol{f}_1.
		\end{align*}
		This gives that
		\begin{align*}
				&(u_0b\boldsymbol{n}_0+ u_2\boldsymbol{n}_2 + v\boldsymbol{n}_1)(\boldsymbol{n}_2\times\boldsymbol{n}_1)^{T}  \\
		+ &(u_0c\boldsymbol{n}_0 + u_3\boldsymbol{n}_3 + w\boldsymbol{n}_1)(\boldsymbol{n}_3\times\boldsymbol{n}_1)^{T} =0 \quad on \ \boldsymbol{f}_1.
		\end{align*}
		Note that $\boldsymbol{n}_2\times \boldsymbol{n}_1,\boldsymbol{n}_3\times \boldsymbol{n}_1$ are linearly independent.\ As a result,
		\begin{align*}
		&u_0b\boldsymbol{n}_0+ u_2\boldsymbol{n}_2 + v\boldsymbol{n}_1 = 0,\\
		&u_0c\boldsymbol{n}_0 + u_3\boldsymbol{n}_3 + w\boldsymbol{n}_1 = 0,\quad on \ \boldsymbol{f}_1.
		\end{align*}
		A similar way as shown before shows that $u_0,u_2,u_3,v,w$ vanish on $\boldsymbol{f}_1$.\ Hence $u_0,u_2,u_3,v,w$ contain a factor $\lambda_1$.
		
		\item[3.]The proof of the case on $\boldsymbol{f}_2$ is similar as above.\ Indeed
		\begin{align*}
		\tau\times\boldsymbol{n}_2&= (au_0\boldsymbol{n}_0 + u_1\boldsymbol{n}_1 +v\boldsymbol{n}_2 + w\boldsymbol{n}_3)(\boldsymbol{n}_1\times\boldsymbol{n}_2)^{T}\\
		& + (cu_0\boldsymbol{n}_0 +u_3\boldsymbol{n}_3 + w\boldsymbol{n}_1)(\boldsymbol{n}_3\times\boldsymbol{n}_2)^{T} = 0\quad on\  \boldsymbol{f}_2.
		\end{align*}
		Therefore,\ a similar argument as before,\ concludes that $u_0,u_1,u_3,v,w$ contain a factor $\lambda_2$.
		
		\item[4.]
		A similar argument on $\boldsymbol{f}_3$ shows that $u_0,u_1,u_2,v,w$ contain a factor $\lambda_3$.
	\end{itemize}
	
	A summary of the above arguments implies that $u_{i}$ contains a factor $\lambda_{j}\lambda_{l}\lambda_{m}$,\ where $\{i,j,l,m\}$ is a permutation of $\{0,1,2,3\}$,\ and $v,w$ contain a factor $\lambda_0\lambda_1\lambda_2\lambda_3$.\ It proves $\tau \in \Sigma_{K,k,b}$,
	which completes the proof.
\end{proof}
\begin{remark}
	It follows from the definition of the bubble function space that,\ if $\sigma\in \Sigma_{\partial K, k, 0}$ for any $k\geq 4$,\ then $D^{\alpha}\sigma(\boldsymbol{x}_{i}) = 0(0\leq i \leq 3)$ for $|\alpha|\leq 1$.
\end{remark}

\subsection{The bubble function space with additional conditions}
For the purpose in the next sections,\ the bubble function space with additional conditions is needed.\ The modified bubble function space consisting of polynomials of degree $\leq k$ is defined by
\begin{equation*}
\Sigma_{\partial K,k,0}^{*} := \{\sigma \in \Sigma_{\partial K,k,0}:D^{\alpha}\sigma({\boldsymbol{x}_{i}}) =0,0 \leq i \leq 3,\forall|\alpha| \leq 2\}.
\end{equation*}

In order to establish a similar theorem as Theorem 3.1,\ define the modified auxiliary space 
\begin{equation*}
P_{k}^{(i,0)}(K;\mathbb{R}) := \{p\in P_{k}^{(i)}(K;\mathbb{R}):p \text{ vanishes at the vertices of }\boldsymbol{f}_{i} \}.
\end{equation*}
Then for $k\geq 4$,\ the bubble function space with additional conditions on $K$ is defined by
\begin{equation}\label{def2.2.1}
\Sigma_{K,k,b}^{*} := \{\sum_{i=0}^{3}\lambda_{j}\lambda_{l}\lambda_{m}P_{k-3}^{(i,0)}(K;\mathbb{R})\boldsymbol{n}_{i}\boldsymbol{n}_{i}^{T}\} \oplus \{ \lambda_0\lambda_1\lambda_2\lambda_3P_{k-4}(K;\mathbb{S}) \}.
\end{equation}
\begin{theorem}
	For $k \geq 4$,\ it holds that
	\begin{equation}\label{Th2.2}
	\Sigma_{\partial K,k,0}^{*} = \Sigma_{K,k,b}^{*}.
	\end{equation}
\end{theorem}
\begin{proof}
	Consider a function $\tau \in \Sigma_{\partial K,k,0}^{*}$,\ which can be expressed as $\tau = \sum_{i=0}^{3}\lambda_{j}\lambda_{l}\lambda_{m}u_{i}\boldsymbol{n}_{i}\boldsymbol{n}_{i}^{T} + \lambda_0\lambda_1\lambda_2\lambda_3\mathbf{S}$ with $u_{i}\in P_{k-3}^{(i)}(K;\mathbb{R})$ and $\mathbf{S}$ a symmetric matrix-valued polynomial of degree $\leq k-4$ with each component,\ says $v$,\ belongs to $P_{k-4}(K;\mathbb{R})$.\ Note that for any function $v\in P_{k-4}(K;\mathbb{R})$,
	\begin{equation}\label{eq2.2.1}
	\begin{split}
	\partial_{i}\partial_{j}(\lambda_0\lambda_1\lambda_2\lambda_3 v) = \partial_{i}\partial_{j}(\lambda_0\lambda_1\lambda_2\lambda_3)v +\partial_{i}(\lambda_0\lambda_1\lambda_2\lambda_3)\partial_{j}v \\ + \partial_{j}(\lambda_0\lambda_1\lambda_2\lambda_3)\partial_{i}v +(\lambda_0\lambda_1\lambda_2\lambda_3) \partial_{i}\partial_{j}v,\quad (1 \leq i,j \leq 3),
	\end{split}
	\end{equation}
	vanishes at all vertices of tetrahedron $K$.\ It implies that $ \sum_{i=0}^{3}\partial_{k}\partial_{t}(\lambda_{j}\lambda_{l}\lambda_{m}u_{i})\boldsymbol{n}_{i}\boldsymbol{n}_{i}^{T},1\leq k,t\leq 3$ vanishes at all the vertices.\ By Lemma 3.1,\ this indicates that $\partial_{k}\partial_{t}(\lambda_{j}\lambda_{l}\lambda_{m}u_{i}) = 0$ for $0\leq i\leq 3$ at all vertices of $K$.\ Similarly,
	\begin{equation}\label{eq2.2.2}
	\begin{split}
	\partial_{k}\partial_{t}(\lambda_{j}\lambda_{l}\lambda_{m} u_{i}) = \partial_{k}\partial_{t}(\lambda_{j}\lambda_{l}\lambda_{m})u_{i} +\partial_{k}(\lambda_{j}\lambda_{l}\lambda_{m})\partial_{t}u_{i}  + \partial_{t}(\lambda_{j}\lambda_{l}\lambda_{m})\partial_{k}u_{i}\\ +(\lambda_{j}\lambda_{l}\lambda_{m}) \partial_{k}\partial_{t}u_{i}\quad (1\leq k,t \leq 3),
	\end{split}
	\end{equation}	 
	where the terms $\partial_{k}(\lambda_{j}\lambda_{l}\lambda_{m}),\partial_{t}(\lambda_{j}\lambda_{l}\lambda_{m})$ and $(\lambda_{j}\lambda_{l}\lambda_{m})$ vanish at all vertices of $K$.\ Since $\partial_{k}\partial_{t}(\lambda_{j}\lambda_{l}\lambda_{m} )$ vanishes at $\boldsymbol{x}_{i}$ and is nonzero at $\boldsymbol{x}_{j},\boldsymbol{x}_{l},\boldsymbol{x}_{m}$,\ which implies that $u_{i}=0$ at $\boldsymbol{x}_{j},\boldsymbol{x}_{l},\boldsymbol{x}_{m}$ and consequently $\tau \in \Sigma_{K,k,b}^{*}$.\ Hence
	\begin{equation*}
	\Sigma_{\partial K,k,0}^{*} \subset \Sigma_{K,k,b}^{*}.
	\end{equation*}
	It remains to show the converse.\ Given $\tau \in \Sigma_{K,k,b}^{*}$,\ it can be expressed as $\tau = \sum_{i=0}^{3}\lambda_{j}\lambda_{l}\lambda_{m}u_{i}\boldsymbol{n}_{i}\boldsymbol{n}_{i}^{T} + \lambda_0\lambda_1\lambda_2\lambda_3  \mathbf{S}$ with $u_{i}\in P_{k-3}^{(i,0)}(K;\mathbb{R})$ and $\mathbf{S}$ defined as above.\ It follows that $D^{\alpha}\tau=0,(|\alpha|=2)$ at all vertices of $K$ from equations \eqref{eq2.2.1},\eqref{eq2.2.2} and a similar argument as above.\ Hence 
	\begin{equation*}
	\tau \in \Sigma_{\partial K,k,0}^{*}.
	\end{equation*}
	This completes the proof.
\end{proof}

Then the $H(\operatorname{curl})$ finite element space with $C^2$ at the vertices is defined by
\begin{equation*}
\begin{split}
\Sigma_{k,h} := \{\sigma \in H(\operatorname{curl},\Omega;\mathbb{S}):\sigma = \sigma_{c} +\sigma_{b},\sigma_{c} \in C^0(\Omega;\mathbb{S}),\\ \sigma_{c}\text{ is }C^2\text{ at } \mathscr{V},
\sigma_{c}\big |_{K}\in P_{k}(K;\mathbb{S}),\sigma_{b}\big |_{K}\in\Sigma_{K,k,b}^{*},\forall K\in \mathcal{T}_{h}\}.
\end{split}
\end{equation*}
Note that this finite element space is a $H(\operatorname{curl})$ bubble enrichment of a $H^1$ space $\widetilde{\Sigma}_{k,h} := \{\tau \in H^1(\Omega;\mathbb{S}),\tau\text{ is }C^2\text{ at } \mathscr{V},\tau\big |_{K}\in P_{k}(K;\mathbb{S}),\forall K\in \mathcal{T}_{h}\}$.

For $k \geq 4$,\ the dimension of the modified bubble function space is
\begin{equation*}
\operatorname{dim} \Sigma_{K,k,b}^{*} = 4[\frac{(k-1)(k-2)}{2} -3] + (k-1)(k-2)(k-3) = k^3-4k^2+5k-14.
\end{equation*}

\subsection{Degrees of freedom}
Assume that $k\geq 5$,\ then the degrees of freedom of the finite element space $\Sigma_{k,h}$ can be given locally as:
\begin{enumerate}
	\item[1.] derivatives of order $\leq 2$ for each component of $\sigma$ at each vertex $\boldsymbol{x}\subset K$:
	\begin{equation*}
	D^{\alpha}\sigma(\boldsymbol{x}) ,\quad \forall |\alpha| \leq 2,
	\end{equation*}
	\item [2.] moments of order $\leq k-6$ for each component of $\sigma$ on each edge $\boldsymbol{e}\subset K$:
	\begin{equation*}
	\int_{\boldsymbol{e}} \sigma : \mathbf{q},\quad \forall \mathbf{q}\in P_{k-6}(\mathbf{e};\mathbb{S}),
	\end{equation*}
	\item [3.]for each face $\boldsymbol{f}_{i} \subset  \partial K(0 \leq i \leq 3)$,\ with the unit normal vector $\boldsymbol{n}_{i}$ and two unit independent tangential vectors $\boldsymbol{\tau}_1,\boldsymbol{\tau}_2$,\ the following moments of $\sigma$: 
	\begin{equation*}
	\int_{\boldsymbol{f}_{i}}\boldsymbol{\tau}_1^{T}\sigma\boldsymbol{\tau}_1 q,\int_{\boldsymbol{f}_{i}}\boldsymbol{\tau}_1^{T}\sigma\boldsymbol{\tau}_2 q,\int_{\boldsymbol{f}_{i}}\boldsymbol{\tau}_2^{T}\sigma\boldsymbol{\tau}_2 q,\int_{\boldsymbol{f}_{i}}\boldsymbol{\tau}_1^{T}\sigma\boldsymbol{n}_{i} q,\int_{\boldsymbol{f}_{i}}\boldsymbol{\tau}_2^{T}\sigma\boldsymbol{n}_{i} q,\quad \forall q\in P_{k-3}^{(i,0)}(\boldsymbol{f}_{i};\mathbb{R})
	\end{equation*}
	\item [4.]the following moments of $\sigma$:
	\begin{equation*}
	\int_{K} \sigma : \mathbf{p},\quad \forall \mathbf{p} \in \Sigma_{K,k,b}^{*}.
	\end{equation*}

\end{enumerate}
In the case $k=5$,\ the second set of degrees of freedom is omitted.
\begin{theorem}
	Any $\sigma \in P_{k}(K;\mathbb{S})$ can be uniquely determined by the local degrees of freedom.

\end{theorem}
\begin{proof}
	The number of local degrees of freedom is:
	$$
	\begin{aligned} & 6\times 10\times 4 + 6\times 6\times (k-5) + 5\times[\frac{(k-1)(k-2)}{2}-3]\times 4 + k^3-4k^2+5k-14 \\
	=&k^3+6k^2 +11k +6 = \dim P_{k}(K;\mathbb{S}).  
	\end{aligned}
	$$
	
	For any $\sigma\in P_{k}(K;\mathbb{S})$,\ it vanishes at all the degrees of freedom if and only if $\sigma=0$ on element $K$.\ In fact from the first set of degrees of freedom,\ the derivatives of $\sigma$ of order $\leq 2$  vanish at all the vertices,\ which indicates that $\sigma$ vanishes at all the edges from the second set of degrees of freedom.\ Furthermore,\  $\sigma|_{\boldsymbol{f}_{i}} = \lambda_{j}\lambda_{l}\lambda_{m}\mathbf{q}$ for some $\mathbf{q}\in P_{k-3}^{(i,0)}(\boldsymbol{f}_{i};\mathbb{S})$.\ Then the third set of degrees of freedom show that $\sigma\times \boldsymbol{n} = 0$ on $\partial K$,\ hence $\sigma \in\Sigma_{\partial K,k,0}^{*}$.\ The fourth set of degrees of freedom and Theorem 3.2 imply that $\sigma = 0$.\ This completes the proof.
\end{proof}

Similarly,\ the local degrees of freedom of $\widetilde{\Sigma}_{k,h}= \{\tau \in H^1(\Omega;\mathbb{S}),\tau\text{ is }C^2\text{ at } \mathscr{V},\tau\big |_{K}\in P_{k}(K;\mathbb{S}),\forall K\in \mathcal{T}_{h}\}$ with $k \geq 5$ are given as follows:
\begin{itemize}
	\item[1.] derivatives of order $\leq 2$ for each component of $\sigma$ at each vertex $\boldsymbol{x}\subset K$:
	\begin{equation*}
	D^{\alpha}\sigma(\boldsymbol{x}) ,\quad \forall |\alpha| \leq 2,
	\end{equation*}
	\item [2.] moments of order $\leq k-6$ for each component of $\sigma$ on each edge $\boldsymbol{e}\subset  K$:
	\begin{equation*}
	\int_{\boldsymbol{e}} \sigma : \mathbf{q},\quad \forall \mathbf{q}\in P_{k-6}(\mathbf{e};\mathbb{S}),
	\end{equation*}
	\item [3.]for each face $\boldsymbol{f}_{i} \subset  \partial K(0 \leq i \leq 3)$,\ the following moments of $\sigma$: 
	\begin{equation*}
	\int_{\boldsymbol{f}_{i}}\sigma:\mathbf{q},\quad \forall\mathbf{q}\in P_{k-3}^{(i,0)}(\boldsymbol{f}_{i};\mathbb{S}),
	\end{equation*}
	\item [4.] moments of order $\leq k-4$ for each component of $\sigma$ on $K$:
	\begin{equation*}
	\int_{K} \sigma : \mathbf{p},\quad \forall \mathbf{p} \in P_{k-4}(K;\mathbb{S}).
	\end{equation*}
\end{itemize}
Note that the first and second sets of degrees of freedom of $\widetilde{\Sigma}_{k,h}$ are the same as those of $\Sigma_{k,h}$,\ therefore the unisolvence of the local degrees of freedom with respect to the shape function space of $\widetilde{\Sigma}_{k,h}$ can be proved similarly as that for $\Sigma_{k,h}$.\ Denote by $\Sigma_{K}$ the set of the linear functionals of the local degrees of freedom on element $K$.\ The nodal interpolation operator is defined by
\begin{equation*}
E_{k}:C^2(\Omega;\mathbb{S})\rightarrow \widetilde{\Sigma}_{k,h},\quad f(E_{k}\sigma) = f(\sigma),\quad \forall f\in \Sigma_{K}.
\end{equation*}
By a scaling argument,\ it holds that
\begin{equation}\label{eq 2.3.1}
\|\sigma-E_{k}\sigma\|_{1,\Omega}\leq Ch^{k}|\sigma|_{k+1,\Omega},\quad \forall \sigma\in C^2(\Omega;\mathbb{S})\cap H^{k+1}(\Omega;\mathbb{S}).
\end{equation}

\section{$H(\operatorname{div})$ Finite Elements for Traceless Matrices}
This section introduces a family of $H(\operatorname{div})$ finite element spaces for traceless matrices.
\subsection{ $\mathbf{H(\operatorname{div})}$ bubble function space}
The full $H(\operatorname{div},K;\mathbb{T})$ bubble function space consisting of polynomials of degree $\leq k$ is defined by
\begin{equation*}
V_{\partial K,k,0} := \{\mathbf{v}\in P_{k}(K;\mathbb{T}):\mathbf{v}  \boldsymbol{n}|_{\partial K} = 0 \},
\end{equation*}
with the unit normal vector $\boldsymbol{n}$ of $\partial K$.

Let $\boldsymbol{\tau}_{i,1},\boldsymbol{\tau}_{i,2}$ be two independent unit tangential vectors on face $\boldsymbol{f}_{i}$.\ Define the following traceless matrices of rank one:
\begin{equation*}
T_{i,j} := \boldsymbol{n}_{i}\boldsymbol{\tau}_{i,j}^{T},\quad 0\leq i\leq 3,1\leq j \leq 2.
\end{equation*}
\begin{lemma}
	The 8 matrices $T_{i,j}$ form a basis of $\mathbb{T}$.
\end{lemma}
\begin{proof}
	It suffices to show that if
	\begin{equation*}
	\sigma = \sum_{i=0}^{3}\sum_{j=1}^{2}c_{i,j}T_{i,j} =0,
	\end{equation*}
	then the constants $c_{i,j}$ are equal to zero.
	
	Note that $\boldsymbol{\tau}_{0,j}^{T}\boldsymbol{n}_0 = 0$,\ this leads to
	\begin{equation*}
	\begin{aligned}
	\sigma  \boldsymbol{n}_0 &= \sum_{i=0}^{3}\sum_{j=1}^{2}c_{i,j}T_{i,j}\boldsymbol{n}_0 = \sum_{i=1}^{3}\sum_{j=1}^{2}c_{i,j}\boldsymbol{n}_{i}\boldsymbol{\tau}_{i,j}^{T}\boldsymbol{n}_0\\
	& = \sum_{i=1}^{3} \widetilde{c}_{i,0}\boldsymbol{n}_{i} = 0,
	\end{aligned}
	\end{equation*}
	where $\widetilde{c}_{i,0}= c_{i,1}\boldsymbol{\tau}_{i,1}^{T}\boldsymbol{n}_{0} + c_{i,2}\boldsymbol{\tau}_{i,2}^{T}\boldsymbol{n}_{0}$.\ Since $\boldsymbol{n}_1,\boldsymbol{n}_2,\boldsymbol{n}_{3}$ are linearly independent,\ it yields $\widetilde{c}_{i,0} = 0$.\ A similar argument leads to
	\begin{equation*}
	\widetilde{c}_{i,j} = c_{i,1}\boldsymbol{\tau}_{i,1}^{T}\boldsymbol{n}_{j} + c_{i,2}\boldsymbol{\tau}_{i,2}^{T}\boldsymbol{n}_{j} =0,\quad 0\leq i\neq j\leq 3.
	\end{equation*}
	Without loss of generality,\ consider the equations for $i=0$
	\begin{equation*}
	\begin{aligned}
	&\widetilde{c}_{0,1} = c_{0,1}\boldsymbol{\tau}_{0,1}^{T}\boldsymbol{n}_{1} + c_{0,2}\boldsymbol{\tau}_{0,2}^{T}\boldsymbol{n}_{1} =0,\\
	&    \widetilde{c}_{0,2} = c_{0,1}\boldsymbol{\tau}_{0,1}^{T}\boldsymbol{n}_{2} + c_{0,2}\boldsymbol{\tau}_{0,2}^{T}\boldsymbol{n}_{2} =0,\\
	&    \widetilde{c}_{0,3} = c_{0,1}\boldsymbol{\tau}_{0,1}^{T}\boldsymbol{n}_{3} + c_{0,2}\boldsymbol{\tau}_{0,2}^{T}\boldsymbol{n}_{3} =0.
	\end{aligned}
	\end{equation*}
	The linear independence of $\boldsymbol{n}_1,\boldsymbol{n}_2,\boldsymbol{n}_3$ gives that $c_{0,1} = c_{0,2}=0$.\ The other constants $c_{i,j}$ are equal to zero by similar arguments.\ This completes the proof.
\end{proof}

Let $\boldsymbol{t}_{i,j}$ denote the tangential vector of edge $\boldsymbol{x}_{i}\boldsymbol{x}_{j}$.\ For $k \geq 3$,\ a $H(\operatorname{div,K;\mathbb{T}})$ bubble function space is defined by
\begin{equation*}
V_{K,k,b} := \sum_{i=0}^{3}\sum_{\substack{0\leq j<l\leq 3\\ j,l\neq i}}\lambda_{j}\lambda_{l}P_{k-2}(K;\mathbb{R})\boldsymbol{n}_{i}\boldsymbol{t}_{j,l}^{T}.
\end{equation*}
The following theorem is similar as Theorem 3.1.
\begin{theorem}
	For $k \geq 3$,\ it holds that
	\begin{equation*}
	V_{\partial K,k,0}= V_{K,k,b}.
	\end{equation*}
\end{theorem}
\begin{proof}
	Consider a function $\mathbf{v} \in  \lambda_{j}\lambda_{l}P_{k-2}(K;\mathbb{R})\boldsymbol{n}_{i}\boldsymbol{t}_{j,l}^{T}$.\ Note that $\mathbf{v}$ vanishes on the faces $\boldsymbol{f}_{j}$ and $\boldsymbol{f}_{l}$.\ For other face which contains edge $\boldsymbol{x}_{j}\boldsymbol{x}_{l}$,\ its unit normal vector $\boldsymbol{n}$ is perpendicular to the tangential vertor $\boldsymbol{t}_{j,l}$ of edge $\boldsymbol{x}_{j}\boldsymbol{x}_{l}$,\ which implies $\mathbf{v}\boldsymbol{n} = 0$ on that face and consequently $\mathbf{v}\in V_{\partial K,k,0}$.\ Hence
	\begin{equation*}
	V_{K,k,b}\subset V_{\partial K,k,0}.
	\end{equation*} 
	It remains to show the converse.\ Given $\mathbf{v} \in V_{\partial K,k,0}$,\ by Lemma 4.1,\ it can be represented as
	\begin{equation}\label{eq 3.1}
	\mathbf{v} = \sum_{i=0}^{3}\sum_{\substack{0\leq j<l\leq 3\\ j,l\neq i}}p_{i}^{j,l}\boldsymbol{n}_{i}\boldsymbol{t}_{j,l}^{T},\quad  p_{i}^{j,l}\in P_{k}(K;\mathbb{R}).
	\end{equation}
	Note that the representation is not unique since the three tangential vectors of the edges on face $\boldsymbol{f}$ are not linearly independent.\ It will be shown that there exists a representation such that  $p_{i}^{j,l}$ contains a factor $\lambda_{j}\lambda_{l}$.
	
	Without loss of  generality,\ consider the case $i=0$.\ On face $\boldsymbol{f}_{1}$,\ it holds that
	\begin{equation*}
	0 = \mathbf{v} \boldsymbol{n}_{1} = \sum_{i=0}^{3}\sum_{\substack{0\leq j<l\leq 3\\ j,l\neq i}}p_{i}^{j,l}\boldsymbol{n}_{i}\boldsymbol{t}_{j,l}^{T}\boldsymbol{n}_{1} = \sum_{i=0,2,3}\sum_{\substack{0\leq j<l\leq 3\\ j,l\neq i}}p_{i}^{j,l}\boldsymbol{n}_{i}(\boldsymbol{t}_{j,l}^{T}\boldsymbol{n}_{1}).
	\end{equation*}
	Since $\boldsymbol{n}_0,\boldsymbol{n}_2,\boldsymbol{n}_3$ are linearly independent,\ this leads to
	\begin{equation*}
	\begin{aligned}
	&0= p_{0}^{1,2}\boldsymbol{t}_{1,2}^{T}\boldsymbol{n}_1 + p_{0}^{2,3}\boldsymbol{t}_{2,3}^{T}\boldsymbol{n}_1 + p_{0}^{1,3}\boldsymbol{t}_{1,3}^{T}\boldsymbol{n}_1 = p_{0}^{1,2}\boldsymbol{t}_{1,2}^{T}\boldsymbol{n}_1 +p_{0}^{1,3}\boldsymbol{t}_{1,3}^{T}\boldsymbol{n}_1,\\
	&0= p_{2}^{0,1}\boldsymbol{t}_{0,1}^{T}\boldsymbol{n}_1 + p_{2}^{0,3}\boldsymbol{t}_{0,3}^{T}\boldsymbol{n}_1 + p_{2}^{1,3}\boldsymbol{t}_{1,3}^{T}\boldsymbol{n}_1 = p_{2}^{0,1}\boldsymbol{t}_{0,1}^{T}\boldsymbol{n}_1 +p_{2}^{1,3}\boldsymbol{t}_{1,3}^{T}\boldsymbol{n}_1,\\
	&0= p_{3}^{0,1}\boldsymbol{t}_{0,1}^{T}\boldsymbol{n}_1 + p_{3}^{0,2}\boldsymbol{t}_{0,2}^{T}\boldsymbol{n}_1 + p_{3}^{1,2}\boldsymbol{t}_{1,2}^{T}\boldsymbol{n}_1 = p_{3}^{0,1}\boldsymbol{t}_{0,1}^{T}\boldsymbol{n}_1 +p_{3}^{1,2}\boldsymbol{t}_{1,2}^{T}\boldsymbol{n}_1,				 
	\end{aligned}
	\quad \text{on}\quad \boldsymbol{f}_1
	\end{equation*}
	which implies that $p_{0}^{1,2}\boldsymbol{t}_{1,2}^{T}\boldsymbol{n}_1 +p_{0}^{1,3}\boldsymbol{t}_{1,3}^{T}\boldsymbol{n}_1$ contains a factor $\lambda_1$.\ A similar argument on faces $\boldsymbol{f}_2,\boldsymbol{f}_3$ shows that $p_{0}^{1,2}\boldsymbol{t}_{1,2}^{T}\boldsymbol{n}_2 +p_{0}^{2,3}\boldsymbol{t}_{2,3}^{T}\boldsymbol{n}_2$ contains a factor $\lambda_2$ and that $p_{0}^{1,3}\boldsymbol{t}_{1,3}^{T}\boldsymbol{n}_3 +p_{0}^{2,3}\boldsymbol{t}_{2,3}^{T}\boldsymbol{n}_3$ contains a factor $\lambda_3$.
	
	Since the representation \eqref{eq 3.1} is not unique,\ $p_{0}^{1,2}$ can be chosen such that it contains a factor $\lambda_1\lambda_2$.\ Note that $\boldsymbol{t}_{i,j}^{T}\boldsymbol{n}_{i} = -\boldsymbol{t}_{j,i}^{T}\boldsymbol{n}_{i} \neq 0,i\neq j$,\ this implies that $p_0^{1,3}$ contains a factor $\lambda_1$ and that $p_{0}^{2,3}$ contains a factor $\lambda_2$.\ In addition $p_0^{1,3}$ and $p_0^{2,3}$ have an expression by $\lambda_1,\lambda_2,\lambda_3$ as follows
	\begin{equation*}
	\begin{aligned}
	&p_0^{1,3} = \lambda_1(p_{1} + \lambda_2 p_2 + \lambda_3 p_3),\\
	&p_0^{2,3} = \lambda_2(q_{1} + \lambda_1 q_2 + \lambda_3 q_3),
	\end{aligned}
	\end{equation*}
	here
	\begin{equation*}
	\begin{aligned}
	&p_1 = \sum_{i=0}^{k-1}c_{1,i}\lambda_1^{i},\quad p_2 = \sum_{i+j \leq k-2}c_{1,i,j}\lambda_1^{i}\lambda_2^{j},\quad p_3\in P_{k-2}(K;\mathbb{R}),\\
	&q_1 = \sum_{i=0}^{k-1}c_{2,i}\lambda_2^{i},\quad q_2 = \sum_{i+j \leq k-2}c_{2,i,j}\lambda_1^{i}\lambda_2^{j},\quad q_3\in P_{k-2}(K;\mathbb{R}).
	\end{aligned}
	\end{equation*}
	Since $\lambda_3\big| p_{0}^{1,3}\boldsymbol{t}_{1,3}^{T}\boldsymbol{n}_3 +p_{0}^{2,3}\boldsymbol{t}_{2,3}^{T}\boldsymbol{n}_3$ and $\boldsymbol{t}_{1,3}^{T}\boldsymbol{n}_3 = \boldsymbol{t}_{2,3}^{T}\boldsymbol{n}_3$,\ this implies that $p_1 = q_1 =0,p_2 = -q_2$.\ This allows to use
	\begin{equation}\label{eq 3.2}
	\widetilde{p}_0^{1,2} = p_0^{1,2} + \lambda_1\lambda_2p_2,\quad \widetilde{p}_0^{1,3} = p_0^{1,3} - \lambda_1\lambda_2p_2,\quad \widetilde{p}_0^{2,3} = p_0^{2,3} - \lambda_1\lambda_2q_2,
	\end{equation}
	to replace $p_0^{1,2},p_0^{1,3},p_0^{2,3}$.\ The equation \eqref{eq 3.2} shows that $\widetilde{p}_0^{i,j}$ contains a factor $\lambda_{i}\lambda_{j}$,\ and that
	\begin{equation*}
	\widetilde{p}_0^{1,2}\boldsymbol{n}_0\mathbf{t}_{1,2}^{T}+\widetilde{p}_0^{1,3}\boldsymbol{n}_0\mathbf{t}_{1,3}^{T}+\widetilde{p}_0^{2,3}\boldsymbol{n}_0\mathbf{t}_{2,3}^{T} = p_0^{1,2}\boldsymbol{n}_0\mathbf{t}_{1,2}^{T}+p_0^{1,3}\boldsymbol{n}_0\mathbf{t}_{1,3}^{T}+p_0^{2,3}\boldsymbol{n}_0\mathbf{t}_{2,3}^{T},
	\end{equation*}
	since $\boldsymbol{t}_{1,2}^{T} + \boldsymbol{t}_{2,3}^{T} = \boldsymbol{t}_{1,3}^{T}$.\ A similar argument shows that other coefficient functions $p_{i}^{j,l}$ can be chosen such that they contain a factor $\lambda_{j}\lambda_{l}$.\ Hence
	\begin{equation*}
	\mathbf{v} \in V_{K,k,b}.
	\end{equation*}
	This completes the proof.
\end{proof}

The next result is concerning the divergence space of the bubble function space.\ Let  $\Xi : \mathbb{M}\rightarrow \mathbb{T}$ be an algebraic operator $\Xi T = T - \frac{1}{3}\operatorname{tr}(T)I$,\ where $I$ is the identity matrix.\ Introduce the following space on element $K$
\begin{equation*}
Q(K) := \{\mathbf{q}\in H^{1}(K;\mathbb{R}^3):\Xi (\nabla \mathbf{q}) = 0\}.
\end{equation*}
For any vector $\mathbf{q}\in Q(K)$,\ it can be written as $\mathbf{q} = (a,b,c)^{T} + d(x,y,z)^{T}$, where $a,b,c,d$ are constants.\ The dimension of $Q(K)$ is 4.\ Define the orthogonal complement space of $Q(K)$ with respect to $P_{k}(K;\mathbb{R}^3)$ by 
\begin{equation*}
Q_{k}^{\perp}(K) := \{\mathbf{q}= (q_1,q_2,q_3)^{T}\in P_{k}(K;\mathbb{R}^3) : \int_{K}q_{i}=0,1\leq i \leq 3,\int_{K}(xq_1+yq_2+zq_3) =0 \}.
\end{equation*}

\begin{theorem}
	For $k \geq 3$,\ it holds that 
	\begin{equation*}
	\operatorname{div} V_{\partial K,k,0} = Q_{k-1}^{\perp}(K).
	\end{equation*}
\end{theorem}
\begin{proof}
	For any $\mathbf{v} \in V_{\partial K,k,0}$,\ it holds that
	\begin{equation*}
	\int_{K} \operatorname{div}\mathbf{v}\cdot \mathbf{q} = \int_{K}\mathbf{v}:\nabla \mathbf{q} = \int_{K}\mathbf{v}:\Xi(\nabla \mathbf{q}) = 0,\quad \forall \mathbf{q}\in Q(K).
	\end{equation*}
	This implies that 
	\begin{equation*}
	\operatorname{div}V_{\partial K,k,0} \subset Q_{k-1}^{\perp}(K).
	\end{equation*}
	To prove the converse,\ assume that $\operatorname{div}V_{\partial K,k,0} \neq Q_{k-1}^{\perp}(K)$.\ Then there exists a nonzero vector $\mathbf{p}\in Q_{k-1}^{\perp}(K)$ such that 
	\begin{equation*}
	\int_{K} \operatorname{div}\mathbf{v}\cdot \mathbf{p} = 0,\quad \forall\mathbf{v}\in V_{\partial K,k,0}.
	\end{equation*}
	By an integration by parts,\ for any $\mathbf{v}\in V_{\partial K,k,0}$,\ it leads to
	\begin{equation*}
	\int_{K} \operatorname{div}\mathbf{v}\cdot \mathbf{p} = \int_{K}\mathbf{v}:\Xi(\nabla \mathbf{p}) = 0.
	\end{equation*}
	Let $\overline{n}$ denote the congruence of integer $n$ with modulus 4,\ namely,
	\begin{equation*}
	\overline{n} \equiv n(\operatorname{mod}4),\quad 0\leq \overline{n} \leq 3.
	\end{equation*}
	By Lemma 4.1,\ it follows that $\boldsymbol{n}_{i}\boldsymbol{t}_{\overline{i+1},\overline{i+2}}^{T},\boldsymbol{n}_{i}\boldsymbol{t}_{\overline{i+1},\overline{i+3}}^{T},0\leq i\leq 3$ form a basis of $\mathbb{T}$.\ Then there exists an associated dual basis $M_{i,j},0\leq i\leq 3,1\leq j\leq 2$,\ such that
	\begin{equation*}
	\boldsymbol{n}_{i}\boldsymbol{t}_{\overline{i+1},\overline{i+1+j}}^{T}:M_{k,l} = \delta_{i,k}\delta_{j,l},\quad 0\leq i,k \leq 3,1\leq j,l \leq 2.
	\end{equation*}
	Since $\Xi (\nabla \mathbf{p})\in P_{k-2}(K;\mathbb{T})$,\ it follows that there exists $p_{i,j} \in P_{k-2}(K;\mathbb{R}),0\leq i \leq 3,1\leq j\leq 2$ such that
	\begin{equation*}
	\Xi (\nabla \mathbf{p}) = \sum_{i=0}^{3}\sum_{j=1}^2M_{i,j}p_{i,j}.
	\end{equation*}
	Then the choice of
	\begin{equation*}
	\mathbf{v} = \sum_{i=0}^{3}\sum_{j=1}^2\lambda_{\overline{i+1}}\lambda_{\overline{i+1+j}}p_{i,j}\boldsymbol{n}_{i}\boldsymbol{t}_{\overline{i+1},\overline{i+1+j}}^{T}\in V_{K,k,b},
	\end{equation*}
	proves
	\begin{equation*}
	\int_{K} \mathbf{v}:\Xi(\nabla \mathbf{p})= \int_{K}\sum_{i=0}^{3}\sum_{j=1}^2\lambda_{\overline{i+1}}\lambda_{\overline{i+1+j}}p_{i,j}^2 =0.
	\end{equation*}
	As $\lambda_{\overline{i+1}}\lambda_{\overline{i+1+j}} \geq 0$ on $K$,\ this implies that $p_{i,j} \equiv 0$,\ and consequently $\Xi(\nabla \mathbf{p})=0$ and $\mathbf{p}\in Q(K)$.\ This contradicts with $\mathbf{p}\in Q_{k-1}^{\perp}(K)$.\ Hence $Q_{k-1}^{\perp}(K)\subset \operatorname{div}V_{\partial K,k,0}$.\ This completes the proof.
\end{proof}
\begin{remark}
	Given $\mathbf{v} \in V_{\partial K,k,0}$,\ by the definition of the bubble function space,\ it implies that $\mathbf{v}$ vanishes at all vertices of element $K$.
\end{remark}
\subsection{$\mathbf{H(\operatorname{div})}$ bubble space with additional conditions}
This subsection considers a $H(\operatorname{div})$ bubble space with additional conditions.\ Define the modified bubble function space $V_{\partial K,k,0}^{*}$ with additional conditions at vertices by
\begin{equation*}
V_{\partial K,k,0}^{*} := \{\mathbf{v}\in V_{\partial K,k,0}:\nabla\mathbf{v}(\boldsymbol{x}_{i}) = 0,0\leq i \leq 3\}.
\end{equation*}
Then the corresponding modified $H(\operatorname{div})$ bubble space $V_{K,k,b}^{*} $ is defined by
\begin{equation*}
V_{K,k,b}^{*}  := \sum_{i=0}^{3}\sum_{\substack{0\leq j<l\leq 3\\ j,l\neq i}}\lambda_{j}\lambda_{l}P_{k-2}^{(j,l,0)}(K;\mathbb{R})\boldsymbol{n}_{i}\boldsymbol{t}_{j,l}^{T},
\end{equation*}
where the auxiliary space $P_{k}^{(i,j,0)}(K;\mathbb{R}) := \{u\in P_{k}(K;\mathbb{R}):u\text{ vanishes at }\boldsymbol{x}_{i},\boldsymbol{x}_{j}\}$.
\begin{theorem}
	For $k \geq 3$,\ it holds that
	\begin{equation*}
	V_{\partial K,k,0}^{*} = V_{K,k,b}^{*}.
	\end{equation*}
\end{theorem}
\begin{proof}
	For any $u \in \lambda_{j}\lambda_{l}P_{k-2}^{(j,l,0)}(K;\mathbb{R})$,\ the gradient $\nabla u$ vanishes at all vertices of $K$.\ By Theorem 4.1,\ it implies that $V_{K,k,b}^{*}\subset V_{\partial K,k,0}^{*} $.
	
	It remains to show the converse.\ Given $\mathbf{v}\in V_{\partial K,k,0}^{*}$,\ by Theorem 4.1,\ it has an expression
	\begin{equation*}
	\mathbf{v} = \sum_{i=0}^{3}\sum_{\substack{0\leq j<l\leq 3\\ j,l\neq i}}\lambda_{j}\lambda_{l}p_{i}^{j,l}\boldsymbol{n}_{i}\boldsymbol{t}_{j,l}^{T},\quad  p_{i}^{j,l}\in P_{k-2}(K;\mathbb{R}).
	\end{equation*}
	Note that
	\begin{equation*}
	\partial_{k}\mathbf{v}(\boldsymbol{x}_{t}) = \sum_{i=0}^{3}\sum_{\substack{0\leq j<l\leq 3\\ j,l\neq i}}\partial_{k}(\lambda_{j}\lambda_{l}p_{i}^{j,l})(\boldsymbol{x}_{t})\boldsymbol{n}_{i}\boldsymbol{t}_{j,l}^{T} = 0,\quad 1 \leq k \leq 3,0 \leq t \leq 3.
	\end{equation*}
	By Lemma 4.1,\ for $ 0 \leq i \leq 3$,
	\begin{equation*}
	\sum_{\substack{0\leq j<l\leq 3\\ j,l\neq i}}\partial_{k}(\lambda_{j}\lambda_{l}p_{i}^{j,l})(\boldsymbol{x}_{t})\boldsymbol{n}_{i}\boldsymbol{t}_{j,l}^{T} = 0,\quad 1 \leq k \leq 3,0 \leq t \leq 3.
	\end{equation*}
	Without loss of generality,\ consider the case where $i=0$ and $t=1$.\ Note that
	\begin{equation*}
	\begin{aligned}
	&\partial_{k} (\lambda_1\lambda_2p_0^{1,2})(\boldsymbol{x}_{1}) =(\lambda_2p_0^{1,2}\partial_{k}\lambda_1 +  \lambda_1p_0^{1,2}\partial_{k}\lambda_2 +\lambda_1\lambda_2\partial_{k}p_0^{1,2})(\boldsymbol{x}_{1}) = (p_0^{1,2}\partial_{k}\lambda_2)(\boldsymbol{x}_1),\\
	&\partial_{k} (\lambda_1\lambda_3p_0^{1,3})(\boldsymbol{x}_{1}) =(\lambda_3p_0^{1,3}\partial_{k}\lambda_1 +  \lambda_1p_0^{1,3}\partial_{k}\lambda_3 +\lambda_1\lambda_3\partial_{k}p_0^{1,3})(\boldsymbol{x}_{1}) = (p_0^{1,3}\partial_{k}\lambda_3)(\boldsymbol{x}_1),\\
	&\partial_{k} (\lambda_2\lambda_3p_0^{2,3})(\boldsymbol{x}_{1}) =(\lambda_3p_0^{2,3}\partial_{k}\lambda_2 +  \lambda_2p_0^{2,3}\partial_{k}\lambda_3 +\lambda_2\lambda_3\partial_{k}p_0^{2,3})(\boldsymbol{x}_{1}) = 0.
	\end{aligned}
	\end{equation*}
	Then 
	\begin{equation*}
	\begin{split}
	\partial_{k} (\lambda_1\lambda_2p_0^{1,2})(\boldsymbol{x}_{1})\boldsymbol{n}_0\boldsymbol{t}_{1,2}^{T} + \partial_{k} (\lambda_1\lambda_3p_0^{1,3})(\boldsymbol{x}_{1})\boldsymbol{n}_0\boldsymbol{t}_{1,3}^{T}+ \partial_{k} (\lambda_2\lambda_3p_0^{2,3})(\boldsymbol{x}_{1})\boldsymbol{n}_0\boldsymbol{t}_{2,3}^{T} \\= 	\partial_{k} \lambda_2p_0^{1,2}(\boldsymbol{x}_{1})\boldsymbol{n}_0\boldsymbol{t}_{1,2}^{T} + \partial_{k} \lambda_3p_0^{1,3}(\boldsymbol{x}_{1})\boldsymbol{n}_0\boldsymbol{t}_{1,3}^{T} = 0,\quad 1 \leq k \leq 3.
	\end{split}
	\end{equation*}
	Since $\boldsymbol{n}_0\boldsymbol{t}_{1,2}^{T}$ and $\boldsymbol{n}_0\boldsymbol{t}_{1,3}^{T}$ are linearly independent,\ this implies that
	\begin{equation*}
	\partial_{k}\lambda_2p_0^{1,2}(\boldsymbol{x}_{1}) = \partial_{k}\lambda_3p_0^{1,3}(\boldsymbol{x}_{1})=0,\quad 1 \leq k \leq 3.
	\end{equation*}
	It holds that $p_0^{1,2}$ and $p_0^{1,3}$ vanish at $\boldsymbol{x}_1$.\ A similar argument shows that $p_{i}^{j,l}$ vanishes at $\boldsymbol{x}_{j},\boldsymbol{x}_{l}$.\ This implies that $\mathbf{v}\in V_{K,k,b}^{*}$.\ Hence
	\begin{equation*}
	V_{\partial K,k,0}^{*} \subset V_{K,k,b}^{*}.
	\end{equation*}
	Which completes the proof.
\end{proof}
Now the $H(\operatorname{div})$ finite element space with $C^1$ at vertices is defined by
\begin{equation*}
\begin{split}
V_{k,h} := \{\mathbf{v} \in H(\operatorname{div},\Omega;\mathbb{T}):\mathbf{v} = \mathbf{v}_{c} +\mathbf{v}_{b},\mathbf{v}_{c} \in C^0(\Omega;\mathbb{T}),\\ \mathbf{v}_{c}\text{ is }C^1\text{ at } \mathscr{V},
\mathbf{v}_{c}\big |_{K}\in P_{k}(K;\mathbb{T}),\mathbf{v}_{b}\big |_{K}\in V_{K,k,b}^{*},\forall K\in \mathcal{T}_{h}\}.
\end{split}
\end{equation*}
This space is a  $H(\operatorname{div})$ bubble space enrichment of the Hermite-type $H^1$ space $\widetilde{V}_{k,h} :=  \{\mathbf{v} \in H^1(\Omega;\mathbb{T}),\mathbf{v}\text{ is }C^1\text{ at } \mathscr{V},\mathbf{v}\big |_{K}\in P_{k}(K;\mathbb{T}),\forall K\in \mathcal{T}_{h}\}$.

For $ k \geq 3$,\ the dimension of the bubble space $V_{\partial K,k,0}^{*}$ is 
\begin{equation*}
\operatorname{dim}V_{\partial K,k,0}^{*} = 4\{3[\frac{(k+1)k(k-1)}{6}-2]-\frac{k(k-1)(k-2)}{6}\} = (4k^3+6k^2-10k-72)/3.
\end{equation*}

It follows from the definition that $V_{\partial K,k,0}^{*}$ is a subspace of $V_{\partial K,k,0}$.\ By Theorem 4.2 the divergence of $V_{\partial K,k,0}^{*}$ belongs to $Q_{k-1}^{\perp}(K)$.\ In fact,\ define the space
\begin{equation*}
R_{k}^{\perp}(K) := \{\mathbf{u}\in Q_{k}^{\perp}(K):\mathbf{u}(\boldsymbol{x}_{i}) = 0,0\leq i \leq 3\}.
\end{equation*}
There holds the following theorem.
\begin{theorem}
	For $k \geq 3$,\ it holds that
	\begin{equation*}
	\operatorname{div} V_{\partial K,k,0}^{*} = R_{k-1}^{\perp}(K).
	\end{equation*}
\end{theorem}
\begin{proof}
	It is obvious that $\operatorname{div} V_{\partial K,k,0}^{*} \subset R_{k-1}^{\perp}(K)$.\ To prove the converse,\ it suffices to show $\operatorname{dim}(\operatorname{div}V_{\partial K,k,0}^{*}) \geq \operatorname{dim}R_{k-1}^{\perp}(K)$.\ First consider the case $k=3$.\ Note that in this case,\ the bubble space $V_{K,3,b}^{*}$ has an explicit representation
	\begin{equation*}
	V_{K,3,b}^{*} = \sum_{\substack{0\leq i,j,l,m \leq 3\\ j\neq l\\ j,l \neq i,m}}P_{0}(K;\mathbb{R})\lambda_{j}\lambda_{l}
	\lambda_{m}\boldsymbol{n}_{i}\boldsymbol{t}_{j,l}^{T}.
	\end{equation*}
	Consider a function $\mathbf{v} = \lambda_{j}\lambda_{l}
	\lambda_{m}\boldsymbol{n}_{i}\boldsymbol{t}_{j,l}^{T} \in V_{K,3,b}^{*}$ with $j \neq l$ and $j,l \neq i,m$,\ it follows that
	\begin{equation*}
	\begin{split}
	\operatorname{div}\mathbf{v} = \operatorname{div}( \lambda_{j}\lambda_{l}
	\lambda_{m}\boldsymbol{n}_{i}\boldsymbol{t}_{j,l}^{T}) = \boldsymbol{n}_{i}\boldsymbol{t}_{j,l}^{T} \nabla(\lambda_{j}\lambda_{l}
	\lambda_{m})
	=  \boldsymbol{n}_{i}\boldsymbol{t}_{j,l}^{T} (\lambda_{j}\lambda_{l}
	\nabla \lambda_{m} \\+ \lambda_{j}\lambda_{m}
	\nabla \lambda_{l} + \lambda_{l}\lambda_{m}
	\nabla \lambda_{j}) =  \boldsymbol{n}_{i}\boldsymbol{t}_{j,l}^{T} ( \lambda_{j}\lambda_{m}
	\nabla \lambda_{l} + \lambda_{l}\lambda_{m}
	\nabla \lambda_{j}).
	\end{split}
	\end{equation*}
	Since $\mathbf{v}\in V_{\partial K,3,0}$,\ it implies that 
	\begin{equation}\label{eq 3.2.1}
	0 = \int_{\partial K}\mathbf{v} \boldsymbol{n} = \int_{K} \operatorname{div}\mathbf{v} = \int_{K}(\boldsymbol{t}_{j,l}^{T}  \nabla \lambda_{l})\lambda_{j}\lambda_{m}\boldsymbol{n}_{i}
	+ (\boldsymbol{t}_{j,l}^{T}
	\nabla \lambda_{j})\lambda_{l}\lambda_{m}\boldsymbol{n}_{i}.
	\end{equation}
	Since 
	$
	\nabla\lambda_{j} =- |\nabla\lambda_{j}|\boldsymbol{n}_{j},
	\nabla\lambda_{l} =- |\nabla\lambda_{l}|\boldsymbol{n}_{l},
	$
	\begin{equation*}
	\boldsymbol{t}_{j,l}^{T}\nabla\lambda_{j} = - \boldsymbol{t}_{j,l}^{T}\nabla\lambda_{l}.
	\end{equation*}
	Then by identity \eqref{eq 3.2.1},\ it follows that
	\begin{equation*}
	\operatorname{div} \mathbf{v} = \boldsymbol{n}_{i}\boldsymbol{t}_{j,l}^{T} ( \lambda_{j}\lambda_{m}
	\nabla \lambda_{l} + \lambda_{l}\lambda_{m}
	\nabla \lambda_{j}) = c(\lambda_{j}\lambda_{m}  - \lambda_{l}\lambda_{m})\boldsymbol{n}_{i}.
	\end{equation*}
	This expression implies that the divergence of the bubble space $V_{K,3,b}^{*}$ has a representation as follows:
	\begin{equation*}
	\operatorname{div}V_{K,3,b}^{*} = \sum_{\substack{0\leq i,j,l,m \leq 3\\ j\neq l\\ j,l \neq i,m}}P_{0}(K;\mathbb{R})(\lambda_{j}\lambda_{m}  - \lambda_{l}\lambda_{m})\boldsymbol{n}_{i}.
	\end{equation*}
	It can be claimed that the dimension of $\operatorname{div}V_{K,3,b}^{*}$ is 14.\ In fact,\ there are 14 vector-valued functions from the divergence space as follows:
	\begin{equation*}
	\begin{aligned}
	\{&\boldsymbol{n}_0(\lambda_1\lambda_0-\lambda_2\lambda_0),\boldsymbol{n}_0(\lambda_2\lambda_0-\lambda_3\lambda_0),\\
	\\
	&\boldsymbol{n}_1(\lambda_0\lambda_1-\lambda_2\lambda_1),\boldsymbol{n}_1(\lambda_2\lambda_1-\lambda_3\lambda_1),\\
	&\boldsymbol{n}_1(\lambda_2\lambda_0-\lambda_3\lambda_0),\boldsymbol{n}_1(\lambda_2\lambda_3-\lambda_0\lambda_3),\\
	\\
	&\boldsymbol{n}_2(\lambda_0\lambda_2-\lambda_1\lambda_2),\boldsymbol{n}_2(\lambda_1\lambda_2-\lambda_3\lambda_2),\\
	&\boldsymbol{n}_2(\lambda_0\lambda_1-\lambda_3\lambda_1),\boldsymbol{n}_2(\lambda_1\lambda_3-\lambda_0\lambda_3),\\
	\\
	&\boldsymbol{n}_3(\lambda_0\lambda_3-\lambda_1\lambda_3),\boldsymbol{n}_3(\lambda_1\lambda_3-\lambda_2\lambda_3),\\
	&\boldsymbol{n}_3(\lambda_0\lambda_1-\lambda_2\lambda_1),\boldsymbol{n}_3(\lambda_1\lambda_0-\lambda_2\lambda_0)\}.
	\end{aligned}
	\end{equation*}
	Since $\boldsymbol{n}_0 = a \boldsymbol{n}_1 + b\boldsymbol{n}_2 +c \boldsymbol{n}_3$ and $\boldsymbol{n}_1, \boldsymbol{n}_2,\boldsymbol{n}_3$ are linearly independent,\ these 14 vectors are linearly independent by a reduction argument.\ This deduces that $\operatorname{dim}(\operatorname{div}V_{\partial K,3,0}^{*}) \geq 14$.\ Note that the dimension of $R_{2}^{\perp}(K)$ is
	\begin{equation*}\operatorname{dim} R_{2}^{\perp}(K) = \operatorname{dim}Q_{2}^{\perp}(K)-12 = 14.
	\end{equation*}
	Hence $\operatorname{div} V_{\partial K,3,0}^{*} = R_{2}^{\perp}(K)$.
	
	Next consider the case $k \geq 3$.\ Define an auxiliary space
	\begin{equation*}
	V_{3} := \sum_{\substack{0 \leq i,j,l \leq 3\\ i \neq j,l\\ j\neq l}}P_0(K;\mathbb{R})\lambda_{j}\lambda_{l}^2\boldsymbol{n}_{i}\boldsymbol{t}_{j,l}^{T}.
	\end{equation*}
	For any
	\begin{equation*}
	\mathbf{v} = \sum_{i=0}^{3}\sum_{\substack{0\leq j<l\leq 3\\ j,l\neq i}}\lambda_{j}\lambda_{l}p_{i}^{j,l}\boldsymbol{n}_{i}\boldsymbol{t}_{j,l}^{T}\in V_{\partial K,k,0},\quad  p_{i}^{j,l}\in P_{k-2}(K;\mathbb{R}),
	\end{equation*}
	there exists a decomposition as follows:
	\begin{equation*}
	\begin{aligned}
	\mathbf{v} &= \sum_{i=0}^{3}\sum_{\substack{0\leq j<l\leq 3\\ j,l\neq i}}\lambda_{j}\lambda_{l}(p_{i}^{j,l} - p_{i}^{j,l}(\boldsymbol{x}_{j})\lambda_{j} - p_{i}^{j,l}(\boldsymbol{x}_{l})\lambda_{l}+ p_{i}^{j,l}(\boldsymbol{x}_{j})\lambda_{j} + p_{i}^{j,l}(\boldsymbol{x}_{l})\lambda_{l}) \boldsymbol{n}_{i}\boldsymbol{t}_{j,l}^{T}\\
	&= \sum_{i=0}^{3}\sum_{\substack{0\leq j<l\leq 3\\ j,l\neq i}}\lambda_{j}\lambda_{l}(p_{i}^{j,l} - p_{i}^{j,l}(\boldsymbol{x}_{j})\lambda_{j} - p_{i}^{j,l}(\boldsymbol{x}_{l})\lambda_{l}) \boldsymbol{n}_{i}\boldsymbol{t}_{j,l}^{T}\\
	&+ \sum_{i=0}^{3}\sum_{\substack{0\leq j<l\leq 3\\ j,l\neq i}}\lambda_{j}\lambda_{l}(  p_{i}^{j,l}(\boldsymbol{x}_{j})\lambda_{j} + p_{i}^{j,l}(\boldsymbol{x}_{l})\lambda_{l}) \boldsymbol{n}_{i}\boldsymbol{t}_{j,l}^{T}.
	\end{aligned}
	\end{equation*}
	Note that
	\begin{equation*}
	\sum_{i=0}^{3}\sum_{\substack{0\leq j<l\leq 3\\ j,l\neq i}}\lambda_{j}\lambda_{l}(p_{i}^{j,l} - p_{i}^{j,l}(\boldsymbol{x}_{j})\lambda_{j} - p_{i}^{j,l}(\boldsymbol{x}_{l})\lambda_{l}) \boldsymbol{n}_{i}\boldsymbol{t}_{j,l}^{T} \in V_{\partial K,k,0}^{*},
	\end{equation*}
	and
	\begin{equation*}
	\sum_{i=0}^{3}\sum_{\substack{0\leq j<l\leq 3\\ j,l\neq i}}\lambda_{j}\lambda_{l}(  p_{i}^{j,l}(\boldsymbol{x}_{j})\lambda_{j} + p_{i}^{j,l}(\boldsymbol{x}_{l})\lambda_{l}) \boldsymbol{n}_{i}\boldsymbol{t}_{j,l}^{T} \in V_3.
	\end{equation*}
	This implies that $V_{\partial K,k,0}\subset V_3 + V_{\partial K,k,0}^{*}$.\ Since $V_{\partial K,k,0}^{*}\subset V_{\partial K,k,0},V_3\subset V_{\partial K,k,0}$,\ this leads to
	\begin{equation*}
	V_{\partial K,k,0}= V_3 + V_{\partial K,k,0}^{*}.
	\end{equation*}
	By Theorem 4.2,\ it holds that
	\begin{equation*}
	\operatorname{div}V_3 + \operatorname{div}V_{\partial K,k,0}^{*} = \operatorname{div}V_{\partial K,k,0} = Q_{k-1}^{\perp}(K).
	\end{equation*}
	Therefore
	\begin{equation*}
	\operatorname{dim}(\operatorname{div}V_3) + \operatorname{dim}(\operatorname{div}V_{\partial K,k,0}^{*}) - \operatorname{dim}(\operatorname{div}V_3 \cap \operatorname{div}V_{\partial K,k,0}^{*}) = \operatorname{dim}Q_{k-1}^{\perp}(K).
	\end{equation*}
	Since it is proved that $\operatorname{div}V_{\partial K,3,0}^{*} = R_{2}^{\perp}(K)$,\ there holds that
	\begin{equation*}
	\begin{aligned}
	\operatorname{dim}(\operatorname{div}V_3) & = \operatorname{dim}Q_{2}^{\perp}(K) - \operatorname{dim}(\operatorname{div}V_{\partial K,3,0}^{*}) + \operatorname{dim}(\operatorname{div}V_3 \cap \operatorname{div}V_{\partial K,3,0}^{*}) \\
	& = \operatorname{dim}Q_{2}^{\perp}(K) - \operatorname{dim}R_{2}^{\perp}(K) + \operatorname{dim}(\operatorname{div}V_3 \cap \operatorname{div}V_{\partial K,3,0}^{*}) \\
	& = 12 + \operatorname{dim}(\operatorname{div}V_3 \cap \operatorname{div}V_{\partial K,3,0}^{*}).
	\end{aligned}
	\end{equation*}
	Note that $V_{\partial K,3,0}^{*}\subset V_{\partial K,k,0}^{*}$ for $k\geq 3$.\ Therefore
	\begin{equation*}
	\begin{aligned}
	\operatorname{dim}(\operatorname{div}V_{\partial K,k,0}^{*}) & = \operatorname{dim}Q_{k-1}^{\perp}(K) + \operatorname{dim}(\operatorname{div}V_3 \cap \operatorname{div}V_{\partial K,k,0}^{*}) - \operatorname{dim}(\operatorname{div}V_3)\\
	& \geq  \operatorname{dim}Q_{k-1}^{\perp}(K) + \operatorname{dim}(\operatorname{div}V_3 \cap \operatorname{div}V_{\partial K,3,0}^{*}) - \operatorname{dim}(\operatorname{div}V_3)\\
	&= \operatorname{dim}Q_{k-1}^{\perp}(K) - 12 = \operatorname{dim}R_{k-1}^{\perp}(K).
	\end{aligned}
	\end{equation*}
	Hence $\operatorname{div}V_{\partial K,k,0}^{*} = R_{k-1}^{\perp}(K)$,\ which completes the proof.
\end{proof}
\subsection{Degrees of freedom}
For $k \geq 3$,\ a unisolvent set of degrees of freedom of the space $V_{k,h}$ is locally given by
\begin{enumerate}
	\item [1.]derivatives of order $\leq 1$ for each component of $\mathbf{v}$ at each $\boldsymbol{x}\subset K$:
	\begin{equation*}
	D^{\alpha}\mathbf{v}(\boldsymbol{x}),\quad \forall |\alpha| \leq 1,
	\end{equation*}
	\item [2.]for each face $\boldsymbol{f}_{i}\subset \partial K(0 \leq i \leq 3)$,\ with the unit normal vector $\boldsymbol{n}_{i}$,\ the following moments of $\mathbf{v}$:
	\begin{equation*}
	\int_{\boldsymbol{f}_{i}}(\mathbf{v}\boldsymbol{n}_{i})\cdot \mathbf{q},\quad \forall \mathbf{q}\in P_{k}(\boldsymbol{f}_{i};\mathbb{R}^3),D^{\alpha}\mathbf{q}\text{ vanish at all vertices of }\boldsymbol{f}_{i},|\alpha| \leq 1,
	\end{equation*}
	\item [3.]the following moments of $\mathbf{v}$:
	\begin{equation*}
	\int_{K}\mathbf{v}:\mathbf{u},\quad \forall \mathbf{u}\in V_{\partial K,k,0}^{*}.
	\end{equation*}
\end{enumerate} 
\begin{theorem}
	Given any element $K$,\ the degrees of freedom are unisolvent for the shape function space $P_{k}(K;\mathbb{T})$.
\end{theorem}
\begin{proof}
	The number of local degrees of freedom is
	\begin{equation*}
	\begin{aligned}
	&8\times 4\times 4 + 4\times 3\times  [\frac{(k+2)(k+1)}{2}-9] + (4k^3+6k^2-10k-72)/3 \\
	=&(4k^3+24k^2+44k+24)/3 = \operatorname{dim}P_{k}(K;\mathbb{T}).
	\end{aligned}
	\end{equation*}
	
	Then for any $\mathbf{v}\in P_{k}(K;\mathbb{T})$,\ it suffices to show $\mathbf{v} = 0$ if all degrees of freedom vanish.\ From the first set of degrees of freedom,\ the derivatives of $\mathbf{v}$ of order $\leq 1$ vanish at all vertices of element $K$.\ Then the second set of degrees of freedom show that $\mathbf{v} \boldsymbol{n} = 0$ on $\partial K$,\ and consequently $\mathbf{v}\in V_{\partial K,k,0}^{*}$.\ The third set of degrees of freedom imply that $\mathbf{v}=0$.\ This proves the unisolvence.
\end{proof}
\section{The discrete Gradgrad-complex}
This section,\ derives and studies the following discrete Gradgrad-complex\eqref{Discrete complex}
\begin{equation*}
P_1(\Omega)\stackrel{\subset}{\longrightarrow}U_{h} \stackrel{\operatorname{gradgrad}}{\longrightarrow}\Sigma_{h} \stackrel{\operatorname{curl}}{\longrightarrow}V_{h}\stackrel{\operatorname{div}}{\longrightarrow}Q_{h}\stackrel{}{\longrightarrow}0
\end{equation*}
with conforming finite element spaces $U_{h} \subset H^2(\Omega;\mathbb{R}),\Sigma_{h}\subset H(\operatorname{curl},\Omega;\mathbb{S}),V_{h}\subset H(\operatorname{div},\Omega;\mathbb{T})$ and $Q_{h}\subset L^2(\Omega;\mathbb{R}^{3})$,\ the exactness of discrete complex will be proved below.

For $k \geq 5$,\ the space $\Sigma_{h}$ is taken as the $H(\operatorname{curl})$ finite element space $\Sigma_{k,h}$ defined above,\ and the $H(\operatorname{div})$ finite element space $V_{h}$ is taken as $V_{k-1,h}$ defined above with degree $k-1$.
\subsection{$H^2$-conforming finite element space}
The space $U_{h}$ is taken as the $H^2$-conforming finite element space defined by Zhang\cite{MR2474112},\ which was first proposed by \v{Z}en\'{\i}\v{s}ek\cite{MR350260}.\ This $C^1$ finite element space consisting of piecewise polynomials of degree $k+2$ with $k \geq 7$,\ which are $C^4$ at the vertices and $C^2$ on edges.\ A unisolvent set of degrees of freedom is given as follows(cf.\cite{MR2474112}):
\begin{enumerate}
	\item [1.]derivatives of $u$ of order $\leq 4$ at each vertex $\boldsymbol{x}\subset K$:
	\begin{equation*}
	D^{\alpha}u(\boldsymbol{x}),\quad \forall |\alpha| \leq 4,
	\end{equation*}
	\item [2.]for each edge $\boldsymbol{e}\subset K$,\ denote by $\boldsymbol{n}$ and $\boldsymbol{m}$ two independent vectors orthogonal to the edge.\ The moments of $u$ of order $\leq k-8$ on $\boldsymbol{e}$,\ moments of the first normal derivatives of $u$ with respect to $\{u_{\boldsymbol{n}},u_{\boldsymbol{m}}\}$ of order $\leq k-7$ on $\boldsymbol{e}$,\ moments of second normal derivatives of $u$ with respect to  $\{u_{\boldsymbol{n}\boldsymbol{n}},u_{\boldsymbol{n}\boldsymbol{m}},u_{\boldsymbol{m}\boldsymbol{m}}\}$ of order $\leq k-6$ on $\boldsymbol{e}$:
	\begin{equation*}
	\int_{\boldsymbol{e}}uq,\quad \forall q\in P_{k-8}(\boldsymbol{e};\mathbb{R}),
	\end{equation*}
	\begin{equation*}
	\int_{\boldsymbol{e}} u_{\boldsymbol{n}}q,\quad \int_{\boldsymbol{e}} u_{\boldsymbol{m}}q,\quad \forall q\in P_{k-7}(\boldsymbol{e};\mathbb{R}),
	\end{equation*}
	\begin{equation*}
	\int_{\boldsymbol{e}}u_{\boldsymbol{n}\boldsymbol{n}}q,\quad \int_{\boldsymbol{e}}u_{\boldsymbol{n}\boldsymbol{m}}q,\int_{\boldsymbol{e}}u_{\boldsymbol{m}\boldsymbol{m}}q,\quad \forall q \in P_{k-6}(\boldsymbol{e};\mathbb{R}),
	\end{equation*}
	in the case $k = 7$, the moments of $u$ on edge is omitted,
	\item [3.]for each face $\boldsymbol{f}\subset \partial K$,\ with the unit normal vector $\boldsymbol{n}$.\ The moments of $u$ of order $\leq k-7$ on $\boldsymbol{f}$,\ and moments of normal derivative $u_{\boldsymbol{n}}$ of order $\leq k-5$ on $\boldsymbol{f}$:
	\begin{equation*}
	\int_{\boldsymbol{f}}uq,\quad \forall q\in P_{k-7}(\boldsymbol{f};\mathbb{R}),
	\end{equation*}
	\begin{equation*}
	\int_{\boldsymbol{f}}u_{\boldsymbol{n}}q,\quad \forall q\in P_{k-5}(\boldsymbol{f};\mathbb{R}),
	\end{equation*}
	\item [4.]moments of $u$ of order $\leq k-6$ on $K$:
	\begin{equation*}
	\int_{K}uq,\quad \forall q\in P_{k-6}(K;\mathbb{R}).
	\end{equation*}
\end{enumerate}
\subsection{$L^2$-conforming finite element space}
The space $Q_{h}$ is taken as the space $Q_{k-2,h}$ consisting of vector-valued piecewise polynomials of degree $k-2$,\ which are continuous at the vertices.\ A unisolvent set of degrees of freedom is given by
\begin{enumerate}
	\item [1.]function value for each component of $\mathbf{q}$ at each vertex $\boldsymbol{x}\subset K$:
	\begin{equation*}
	\mathbf{q}(\boldsymbol{x}),
	\end{equation*}
	\item [2.]the following moments of $\mathbf{q}$:
	\begin{equation*}
	\int_{K}\mathbf{q}\cdot\mathbf{u},\quad \forall \mathbf{u}\in P_{k-2}(K;\mathbb{R}^3),\mathbf{u}\text{ vanishes at the vertices of }K.
	\end{equation*}
\end{enumerate}
It is obvious that these degrees of freedom uniquely determine a function of $Q_{k-2,h}$ on each element $K$.
\subsection{discrete complex}
Following from the definition of spaces,\ it holds that $\operatorname{gradgrad} U_{h}\subset \Sigma_{h},\operatorname{curl}\Sigma_{h}\subset V_{h}$ and $\operatorname{div}V_{h}\subset Q_{h}$.\ Next it will be shown that this complex is exact.

For any $\sigma\in\Sigma_{k,h}$ satisfying $\operatorname{curl}\sigma=0$ with $k \geq 7$,\ there exists $u\in H^2(\Omega;\mathbb{R})$ such that $\sigma = \operatorname{grad}\operatorname{grad}u$.\ Since $\sigma$ is a matrix-valued piecewise polynomial of degree $\leq k$,\ $u$ is a piecewise polynomial of degree $\leq k+2$.\ It follows from the definition of $\Sigma_{k,h}$ that $u$ is $C^4$ at all vertices and $C^2$ on edges,\ it implies that $u\in U_{h}$.

The next theorem shows that the divergence operator $\operatorname{div}:V_{k-1,h}\rightarrow Q_{k-2,h}$ is onto.
\begin{theorem}
	For any $\mathbf{q}_{h}\in Q_{k-1,h}$ with $k \geq 4$,\ there is a $\mathbf{v}_{h}\in V_{k,h}$ such that $\operatorname{div}\mathbf{v}_{h} = \mathbf{q}_{h}$ and $\|\mathbf{v}_{h}\|_{H(\operatorname{div},\Omega)}\leq C\|\mathbf{q}_{h}\|_{0,\Omega}$.
\end{theorem}
\begin{proof}
	Given $\mathbf{q}_{h} = (q_1,q_2,q_3)^{T}\in Q_{k-1,h}$,\ there exists $\mathbf{v}\in H^1(\Omega;\mathbb{T})$ such that $\operatorname{div}\mathbf{v} = \mathbf{q}_{h}$ and $\|\mathbf{v}\|_{1,\Omega} \leq C\|\mathbf{q}_{h}\|_{0,\Omega}$\cite{MR4113080}.\ Let $I_{h}$ denote the Scott-Zhang interpolation operator which satisfies that\cite{MR1011446}
	\begin{equation*}
	\|\mathbf{v}-I_{h}\mathbf{v}\|_{0,\Omega} + h\|\nabla I_{h}\mathbf{v}\|_{0,\Omega} \leq Ch\|\nabla\mathbf{v}\|_{0,\Omega}.
	\end{equation*}
	The construction of $\mathbf{v}_{h}$ will be finished in three steps.\ The first step defines $\mathbf{v}_1 = (v_{i,j})_{3\times 3}\in V_{k,h}$ such that
	\begin{itemize}
		\item [1.]for each vertex $\boldsymbol{x}\in\mathscr{V}$
		\begin{equation*}
		\mathbf{v}_1(\boldsymbol{x}) = I_{h}\mathbf{v}(\boldsymbol{x}),\quad \frac{\partial v_{i,i}}{\partial x_{j}}(\boldsymbol{x}) = 0,\quad \frac{\partial v_{i,j}}{\partial x_{j}}(\boldsymbol{x}) = \frac{1}{2}q_{i}(x)(i \neq j),\quad \frac{\partial v_{i,j}}{\partial x_{k}}(\boldsymbol{x}) = 0(i ,k\neq j),
		\end{equation*}
		\item [2.]for each face $\boldsymbol{f}\in \mathscr{F}$ with the unit normal vector $\boldsymbol{n}$
		\begin{equation*}
		\int_{\boldsymbol{f}}(\mathbf{v}_1-I_{h}\mathbf{v})\boldsymbol{n}\cdot \mathbf{q} = 0,\quad \forall \mathbf{q}\in P_{k}(\boldsymbol{f};\mathbb{R}^3),\text{ with }D^{\alpha}\mathbf{q}\text{ vanishing at all vertices of }\boldsymbol{f},|\alpha| \leq 1,
		\end{equation*}
		\item [3.]
		for any element $K\in \mathscr{T}$		
		\begin{equation*}
		\int_{K}\mathbf{v}_1:\mathbf{u}=  \int_{K}I_{h}\mathbf{v}:\mathbf{u},\quad \forall\mathbf{u}\in V_{\partial K,k,0}^{*}.
		\end{equation*}
	\end{itemize}
	
	Since $k\geq 4$,\ for each face $\boldsymbol{f}\in\mathscr{F}$ of element $K$,\ there are $\frac{(k-1)(k-2)}{2} \geq 3$ Hermite type basis functions $\phi_{l}\in \{p \in H^1(\Omega;\mathbb{R}),p\text{ is }C^1\text{ at } \mathscr{V},p\big |_{K}\in P_{k}(K;\mathbb{R}),\forall K\in \mathcal{T}_{h}\},1 \leq l \leq \frac{(k-1)(k-2)}{2}$, such that $\phi_{l}$ vanish on $\partial(K^{+}\cup K^{-})$ and $\nabla\phi_{l}$ vanish at the vertices of $K^{+}$ and $K^{-}$,\ where $K^{+}$ and $K^{-}$ are two elements that share the face $\boldsymbol{f}$.\ Then $\phi_{l}T_{i,j},0\leq i\leq3,1\leq j \leq 2,1 \leq l \leq \frac{(k-1)(k-2)}{2}$ are matrix-valued functions,\ which are linearly independent.\ These bubble functions allow us to define a correction $\delta_{h}\in V_{k,h}$ such that 
	\begin{equation*}
	\int_{\boldsymbol{f}}\delta_{h}\boldsymbol{n}\cdot \mathbf{q} = \int_{\boldsymbol{f}}(\mathbf{v}-\mathbf{v}_1)\boldsymbol{n}\cdot \mathbf{q},\quad \forall \mathbf{q}\in Q(K)|_{\boldsymbol{f}}.
	\end{equation*}
	Then the second step defines
	\begin{equation*}
	\widetilde{\mathbf{v}}_1 = \mathbf{v}_1 +\delta_{h}.
	\end{equation*}
	By the construction,\ $\operatorname{div}\widetilde{\mathbf{v}}_1(\boldsymbol{x})=\mathbf{q}_{h}(\boldsymbol{x})$ at all of the vertices.\ On each element $K$,\ an integration by parts,\ for any $\mathbf{q}\in Q(K)$,\ yields
	\begin{equation*}
	\begin{aligned}
	\int_{K}(\operatorname{div}\widetilde{\mathbf{v}}_1-\mathbf{q}_{h})\cdot\mathbf{q} & = \int_{K}(\operatorname{div}\widetilde{\mathbf{v}}_1-\operatorname{div}\mathbf{v})\cdot\mathbf{q}\\
	& = \int_{\partial K}(\widetilde{\mathbf{v}}_1-\mathbf{v})\boldsymbol{n}\cdot\mathbf{q} = 0.
	\end{aligned}
	\end{equation*}
	This implies that $(\operatorname{div}\widetilde{\mathbf{v}}_1-\mathbf{q}_{h})|_{K} \in R_{k-1}^{\perp}(K)$.\ It follows from Theorem 4.4 that there exists a $\mathbf{v}_2\in V_{k,h}$ such that $\mathbf{v}_2|_{K}\in V_{\partial K,k,0}^{*}$ and 
	\begin{equation*}
	\operatorname{div}\mathbf{v}_2 = \mathbf{q}_{h}-\operatorname{div}\widetilde{\mathbf{v}}_1,\quad \|\mathbf{v}_2\|_{0,\Omega} = \min\{\|\mathbf{v}\|_{0,\Omega},\operatorname{div}\mathbf{v} = \mathbf{q}_{h}-\operatorname{div}\widetilde{\mathbf{v}}_1,\mathbf{v}\in V_{k,h}\}.
	\end{equation*}
	Note that $\|\operatorname{div}\mathbf{v}_2\|_{0,\Omega}$ defines a norm of $\mathbf{v}_2$.\ Then a scaling argument shows that
	\begin{equation*}
	\|\mathbf{v}_2\|_{H(\operatorname{div},\Omega)} \leq C\|\operatorname{div}\mathbf{v}_2\|_{0,\Omega}.
	\end{equation*}
	The third step defines $\mathbf{v}_{h} = \widetilde{\mathbf{v}}_1+\mathbf{v}_2$.\ It follows that $\operatorname{div}\mathbf{v}_{h} = \mathbf{q}_{h}$ and 
	\begin{equation*}
	\|\mathbf{v}_{h}\|_{H(\operatorname{div},\Omega)} \leq \|\widetilde{\mathbf{v}}_{1}\|_{H(\operatorname{div},\Omega)} + \|\mathbf{v}_{2}\|_{H(\operatorname{div},\Omega)}\leq C(\|\widetilde{\mathbf{v}}_{1}\|_{H(\operatorname{div},\Omega)} + \|\mathbf{q}_{h}-\operatorname{div}\widetilde{\mathbf{v}}_1\|_{0,\Omega}).
	\end{equation*}
	For each element $K\in \mathcal{T}_{h}$,\ let $\omega(K)=\sum_{K^{'}\in\mathcal{T}_{h}, \overline{K^{'}}\cap \overline{K}\neq \varnothing} K^{'}$ denote the patch of $K$.\ A standard scaling argument and the trace theory,\ lead to
	\begin{equation*}
	\|\mathbf{v}_1-I_{h}\mathbf{v}\|_{1,K} \leq C(\|\mathbf{v}\|_{1,\omega
		(K)} + \|\mathbf{q}_{h}\|_{0,K}),
	\end{equation*}
	\begin{equation*}
	\|\delta_{h}\|_{1,K} \leq C\|\mathbf{v}-\mathbf{v}_1\|_{1,K},
	\end{equation*}
	this implies that $\|\widetilde{\mathbf{v}}_1\|_{H(\operatorname{div},\Omega)} \leq C\|\mathbf{q}_{h}\|_{0,\Omega}$,\ and consequently $\|\mathbf{v}_{h}\|_{H(\operatorname{div},\Omega)}\leq C\|q_{h}\|_{0,\Omega}$.
\end{proof}

The exactness of the discrete complex \eqref{Discrete complex} is proved in the following theorem.
\begin{theorem}
	The discrete complex \eqref{Discrete complex} is exact with $k \geq 7$.
\end{theorem}
\begin{proof}
	It suffices to check the dimension.\ The degrees of freedom of $U_{h}$ given above show that the global dimension of $U_{h}$ is 
	$35\mathcal{V} + (6k-34)\mathcal{E} + (k^2-9k+21)\mathcal{F} + \frac{1}{6}(k-3)(k-4)(k-5)\mathcal{T} $.\ Similarly,\ by the degrees of freedom defined above,\ the dimension of  $\Sigma_{h}$ is $60\mathcal{V} + 6(k-5)\mathcal{E} + \frac{5(k^2-3k-4)}{2}\mathcal{F} + (k^3-4k^2+5k-14)\mathcal{T}$,\ the dimension of $V_{h}$ is $32\mathcal{V} + \frac{3(k^2+k-18)}{2}\mathcal{F} +  \frac{4k^3-6k^2-10k-60}{3}\mathcal{T}$ and the dimension of $Q_{h}$ is $3\mathcal{V} + \frac{k^3-k-24}{2}\mathcal{T}$.
	
	By the Euler's formula $\mathcal{V}-\mathcal{E} +\mathcal{F}-\mathcal{T} =1$,\ it holds that
	\begin{equation*}
	\operatorname{dim}U_{h}-\operatorname{dim}\Sigma_{h}+\operatorname{dim}V_{h}-\operatorname{dim}Q_{h}  = 4 = \operatorname{dim}P_1(\Omega).
	\end{equation*}
	This completes the proof.
\end{proof}
\begin{remark}
	Consider another complex
	\begin{equation}\label{Complex2}
	H^2(\Omega;\mathbb{R}) \stackrel{\operatorname{gradgrad}}{\longrightarrow}H(\operatorname{curl},\Omega;\mathbb{S}) \stackrel{\operatorname{curl}}{\longrightarrow}L^2(\Omega;\mathbb{T}),
	\end{equation}
	which is closed but not exact.\ Define the full $C^{-1}$-$P_{k-1}$ space
	\begin{equation*}
	\widehat{V}_{h}:= \{\mathbf{v}\in L^2(\Omega;\mathbb{T}):\mathbf{v}|_{K}\in P_{k-1}(K;\mathbb{T}),\forall K\in\mathcal{T}_{h}\}.
	\end{equation*}
	There holds the following discrete sub-complex of \eqref{Complex2}
	\begin{equation*}
	U_{h} \stackrel{\operatorname{gradgrad}}{\longrightarrow}\Sigma_{h} \stackrel{\operatorname{curl}}{\longrightarrow}\widehat{V}_{h}.
	\end{equation*}	
\end{remark}
\section{Mixed Methods of The Linearized Einstein-Bianchi System with Strong Symmetry}
\subsection{Mixed methods}
This section considers the mixed approximation of the new formulation of the linearized Einstein-Bianchi system introduced in\cite{MR3450102}:\ Find
\begin{equation}\label{eq5.0}
\begin{aligned} & \sigma \in C^{0}([0, T], {H}^{2}(\Omega;\mathbb{R})) \cap C^{1}([0, T], L^{2}(\Omega;\mathbb{R})), \\ &\mathbf{E} \in C^{0}([0, T], {H}(\operatorname{curl},\Omega; \mathbb{S})) \cap C^{1}([0, T], L^{2}(\Omega;\mathbb{S})), \\ &\mathbf{B} \in  C^{1}([0, T], L^{2}(\Omega;\mathbb{T})), \end{aligned}
\end{equation}
such that
\begin{equation}\label{eq5.1}
\begin{cases} (\dot{\sigma},\tau) = (\mathbf{E},\operatorname{grad}\operatorname{grad}\tau),&\forall\tau\in {H}^2(\Omega;\mathbb{R}), \\ (\dot{\mathbf{\mathbf{E}}},\mathbf{u}) =-(\operatorname{grad} \operatorname{grad} \sigma,\mathbf{u})-  (\mathbf{B},\operatorname{curl}\mathbf{u}),& \forall \mathbf{u} \in {H}(\operatorname{curl},\Omega;\mathbb{S}), \\ (\dot{\mathbf{B}},\mathbf{v}) =(\operatorname{curl} \mathbf{E},\mathbf{v}),& \forall \mathbf{v} \in L^2(\Omega;\mathbb{T}), \end{cases}
\end{equation}
with given initial data 
$
(\sigma(0), \mathbf{E}(0), \mathbf{B}(0))\in {H}^{2}(\Omega;\mathbb{R}) \times {H}(\operatorname{curl},\Omega; \mathbb{S})  \times L^2(\Omega;\mathbb{T})
$.

For $k\geq 7$,\ the semidiscretization of (\ref{eq5.1}) is to find
$$\sigma_{h} \in C^{1}\left(\left[0, T_{0}\right] , U_{h}\right),\quad  \mathbf{E}_{h} \in C^{1}\left(\left[0, T_{0}\right] , \Sigma_{h}\right),\quad  \mathbf{B}_{h} \in C^{1}(\left[0, T_{0}\right] , \widehat{V}_{h}),$$ 
such that
\begin{equation*}
\begin{cases} (\dot{\sigma_{h}},\tau) = (\mathbf{E}_{h},\operatorname{grad}\operatorname{grad}\tau),& \forall\tau\in U_{h}, \\ (\dot{\mathbf{E}_{h}},\mathbf{u}) =-(\operatorname{grad} \operatorname{grad} \sigma_{h},\mathbf{u})-  (\mathbf{B}_{h},\operatorname{curl}\mathbf{u}),&\forall \mathbf{u} \in \Sigma_{h}, \\ (\dot{\mathbf{B}_{h}},\mathbf{v})  =(\operatorname{curl} \mathbf{E}_{h},\mathbf{v}),&\forall \mathbf{v} \in \widehat{V}_{h}, \end{cases}
\end{equation*}
for all time $t\in [0,T]$,\ with given initial data.
\begin{theorem}
	There exists a unique solution for the above semidiscrete system .
\end{theorem}
\begin{proof}
	Let $\{\psi_{i}\},\{\varphi_{i}\},\{\chi_{i}\}$ denote the bases of $U_{h},\Sigma_{h},\widehat{V}_{h}$,\ respectively.\ Then the variables $\sigma_{h},\mathbf{E}_{h}$ and $\mathbf{B}_{h}$ can be expressed as $\sigma_{h} = \sum_{i}\alpha_{i}(t)\psi_{i},\boldsymbol{E}_{h} = \sum_{i}\beta_{i}(t)\varphi_{i},\boldsymbol{B}_{h} = \sum_{i}\gamma_{i}(t)\chi_{i}$ with $\alpha_{i}(t),\beta_{i}(t)$ and $\gamma_{i}(t)$ are coefficient functions with respect to $t$.\ Let $\alpha,\beta,\gamma$ denote the corresponding vectors.\ Let $\mathscr{A},\mathscr{B},\mathscr{C},\mathscr{M},\mathscr{N}$ denote the matrices whose $(i,j)$-entries are
	\begin{equation*}
	(\psi_{j},\psi_{i}),\quad (\varphi_{j},\varphi_{i}),\quad (\chi_{j},\chi_{i}),\quad (\varphi_{j},\operatorname{grad}\operatorname{grad}\psi_{i}),\quad (\operatorname{curl}\varphi_{j},\chi_{i}),
	\end{equation*}
	respectively.\ Then the semidiscrete system can be written in a matrix equation form
	$$\left(\begin{array}{ccc}
	\mathscr{A} & 0 & 0 \\
	0 & \mathscr{B} & 0 \\
	0 & 0 & \mathscr{C}
	\end{array}\right)\left(\begin{array}{c}
	\dot{\alpha} \\
	\dot{\beta} \\
	\dot{\gamma}
	\end{array}\right)=\left(\begin{array}{ccc}
	0 & \mathscr{M} &  0\\
	-\mathscr{M}^{T} & 0 & -\mathscr{N}^{T} \\
	0 & \mathscr{N} & 0
	\end{array}\right)\left(\begin{array}{l}
	\alpha \\
	\beta \\
	\gamma
	\end{array}\right)$$
	The above system is a linear system of ordinary differential equations.\ Note that the coefficient matrix on the left-hand side is nonsingular.\ By the ODE theory,\ the equation is well-posed as an initial value problem,\ so there exists a unique solution.
\end{proof}
In order to get the convergence of the discrete solutions,\ consider the following elliptic problem.\ Let $\mathbf{V}_{h} = U_{h} \times  \Sigma_{h}\times \widehat{V}_{h}$ with norm 
\begin{equation*}
\|(\sigma,\mathbf{E},\mathbf{B})\|_{\mathbf{V}_{h}} = \|\sigma\|_{2,\Omega} + \|\mathbf{E}\|_{H(\operatorname{curl},\Omega)} +\|\mathbf{B}\|_{0,\Omega}.
\end{equation*}
Define the bilinear form $A : \mathbf{V}_{h} \times \mathbf{V}_{h} \rightarrow \mathbb{R}$ by
\begin{equation*}
\begin{array}{c}
A(\sigma,\mathbf{E},\mathbf{B};\tau,\mathbf{u},\mathbf{v}) = (\sigma,\tau) + (\mathbf{E},\mathbf{u}) + (\mathbf{B},\mathbf{v}) -(\mathbf{E},\operatorname{grad}\operatorname{grad}\tau)\\
+(\operatorname{grad} \operatorname{grad} \sigma,\mathbf{u})+ (\mathbf{B},\operatorname{curl}\mathbf{u}) - (\operatorname{curl} \mathbf{E},\mathbf{v}),
\end{array}
\end{equation*}
which is uniformly bounded independent of $h$.\ The next theorem and proof is adapted from\cite{MR3450102}.

\begin{theorem}
	The bilinear form $A(\cdot,\cdot,\cdot;\cdot,\cdot,\cdot)$ defined above satisfies the inf-sup condition:
	$$
	\inf _{0 \neq  (\sigma,\mathbf{E},\mathbf{B})\in \mathbf{V}_{h}} \sup _{0 \neq (\tau,\mathbf{u},\mathbf{v}) \in \mathbf{V}_{h}} \frac{A(\sigma,\mathbf{E},\mathbf{B};\tau,\mathbf{u},\mathbf{v} )}{\|(\sigma,\mathbf{E},\mathbf{B})\|_{\mathbf{V}_{h}}\|(\tau,\mathbf{u},\mathbf{v})\|_{\mathbf{V}_{h}}}=C >0 
	$$
	with constant $C$ independent of $h$.
\end{theorem}
\begin{proof}
	For any $(\sigma,\mathbf{E},\mathbf{B}) \in \mathbf{V}_{h}$,\ let $(\tau,\mathbf{u},\mathbf{v}) = (\sigma,\mathbf{E} + \operatorname{grad}\operatorname{grad}\sigma,\mathbf{B} - \operatorname{curl} \mathbf{E})$.\ It follows that
	\begin{equation*}
	\begin{aligned}
	A(\sigma,\mathbf{E},\mathbf{B};\tau,\mathbf{u},\mathbf{v} ) = &A(\sigma,\mathbf{E},\mathbf{B};\sigma,\mathbf{E},\mathbf{B}) + A(\sigma,\mathbf{E},\mathbf{B};0,\operatorname{grad}\operatorname{grad}\sigma,- \operatorname{curl} \mathbf{E})\\
	= &(\sigma,\sigma) +(\mathbf{E},\mathbf{E})+ (\mathbf{B},\mathbf{B}) + (\mathbf{E},\operatorname{grad}\operatorname{grad}\sigma) \\
	&+ (\operatorname{grad}\operatorname{grad}\sigma,\operatorname{grad}\operatorname{grad}\sigma)- (\mathbf{B},\operatorname{curl} \mathbf{E}) + (\operatorname{curl} \mathbf{E},\operatorname{curl} \mathbf{E})\\
	\geq & \frac{1}{2}(\|\sigma\|_{0,\Omega}^2 + \|\mathbf{E}\|_{0,\Omega}^2 + \|\mathbf{B}\|_{0,\Omega}^2 + \|\operatorname{grad}\operatorname{grad}\sigma\|_{0,\Omega}^2 + \|\operatorname{curl} \mathbf{E}\|^2_{0,\Omega} )\\
	\geq&C\|(\sigma,\mathbf{E},\mathbf{B})\|_{\mathbf{V}_{h}}^2.
	\end{aligned}
	\end{equation*}
	Since $\|(\sigma,\mathbf{E},\mathbf{B})\|_{\mathbf{V}_{h}} \geq C \|(\tau,\mathbf{u},\mathbf{v})\|_{\mathbf{V}_{h}}$,\ this implies that
	\begin{equation*}
	A(\sigma,\mathbf{E},\mathbf{B};\tau,\mathbf{u},\mathbf{v}) \geq C\|(\sigma,\mathbf{E},\mathbf{B})\|_{\mathbf{V}_{h}}\|(\tau,\mathbf{u},\mathbf{v})\|_{\mathbf{V}_{h}}.
	\end{equation*}
	This completes the proof.
\end{proof}

Next,\ the inf-sup condition will be used to analyze the well-posedness of the discrete problem.\ To this end,\ for any $(\sigma,\mathbf{E},\mathbf{B})\in {H}^{2}(\Omega;\mathbb{R}) \times {H}(\operatorname{curl},\Omega; \mathbb{S}) \times L^2(\Omega;\mathbb{T})$,\ define the elliptic projection $\Pi_{h}(\sigma,\mathbf{E},\mathbf{B})\in\mathbf{V}_{h}$ such that 
\begin{equation}\label{eq5.2}
A(\Pi_{h}\sigma,\Pi_{h}\mathbf{E},\Pi_{h}\mathbf{B};\tau,\mathbf{u},\mathbf{v}) = A(\sigma,\mathbf{E},\mathbf{B};\tau,\mathbf{u},\mathbf{v} ),\qquad \forall (\tau,\mathbf{u},\mathbf{v})\in \mathbf{V}_{h}.
\end{equation}
There holds the following quasioptimal error estimate:
\begin{equation}\label{eq5.1.1}
\begin{split}
\|\Pi_{h}\sigma-\sigma\|_{2,\Omega}+\|\Pi_{h}\mathbf{E}-\mathbf{E}\|_{H(\operatorname{curl},\Omega)}+\|\Pi_{h}\mathbf{B}-\mathbf{B}\|_{0,\Omega} \\ \leq C \inf_{(\tau,\mathbf{u},\mathbf{v})\in \mathbf{V}_{h}}\|\tau -\sigma\|_{2,\Omega}+\|\mathbf{u}-\mathbf{E}\|_{H(\operatorname{curl},\Omega)}+\|\mathbf{v}-\mathbf{B}\|_{0,\Omega}.
\end{split}
\end{equation}
Let $P_{h}:L^2(\Omega;\mathbb{T})\rightarrow \widehat{V}_{h}$ denote the piecewise $L^2$ projection operator.\ Its error estimate is as follows:
\begin{equation}\label{eq5.1.2}
\|\mathbf{v}- P_{h}\mathbf{v}\|_{0,\Omega} \leq Ch^{k}\|\mathbf{v}\|_{k,\Omega},\qquad \forall \mathbf{v}\in H^{k}(\Omega;\mathbb{T}).
\end{equation}
Note that $E_{k}\mathbf{u}\in \Sigma_{h}$ for any $\mathbf{u}\in C^2(\Omega;\mathbb{S})$ where $E_{k}$ is defined in \eqref{eq 2.3.1}.\ The error estimate is as follows:
\begin{equation}\label{eq5.1.3}
\|\mathbf{u} - E_{k}\mathbf{u}\|_{1,\Omega}  \leq Ch^{k}\|\mathbf{u}\|_{k+1,\Omega},\quad \forall \mathbf{u}\in H^{k+1}(\Omega;\mathbb{S})\cap C^2(\Omega;\mathbb{S}).
\end{equation}
Let $\Pi_{k+2}$ be the nodal interpolation operator for the space $U_{h}$ defined in \cite{MR2474112}.\ It holds the following error estimate:
\begin{equation}\label{eq5.1.4}
\|\tau - \Pi_{k+2}\tau\|_{2,\Omega} \leq Ch^{k+1}\|\tau\|_{k+3,\Omega},\qquad \forall \tau \in H^{k+3}(\Omega;\mathbb{R})\cap C^4(\Omega;\mathbb{R}).
\end{equation}
Assume that $\sigma\in H^{k+3}(\Omega;\mathbb{R})\cap C^4(\Omega;\mathbb{R}),\mathbf{E}\in H^{k+1}(\Omega;\mathbb{S})\cap C^2(\Omega;\mathbb{S})$,\ and $\mathbf{B}\in H^{k}(\Omega;\mathbb{T})$,\ let $\tau = \Pi_{k+2}\sigma,\mathbf{u} = E_{k}\mathbf{E},\mathbf{v} = P_{h}\mathbf{B}$ in (\ref{eq5.1.1}),\ by (\ref{eq5.1.2}),\ (\ref{eq5.1.3}) and (\ref{eq5.1.4}).\ A standard argument shows that
\begin{equation}\label{eq3.1.5}
\begin{split}
\|\Pi_{h}\sigma-\sigma\|_{2,\Omega}+\|\Pi_{h}\mathbf{E}-\mathbf{E}\|_{H(\operatorname{curl},\Omega)}+\|\Pi_{h}\mathbf{B}-\mathbf{B}\|_{0,\Omega} \\ \leq C h^{k}(\|\sigma\|_{k+3,\Omega} + \|\mathbf{E}\|_{k+1,\Omega} + \|\mathbf{B}\|_{k,\Omega}).
\end{split}
\end{equation}
Note that equation (\ref{eq5.2}) admits another equivalent form as follows:
\begin{equation}\label{eq5.3}
\begin{cases}
(\Pi_{h} \sigma,\tau) -(\Pi_{h}\mathbf{E},\operatorname{grad}\operatorname{grad}\tau) = (\sigma,\tau) -(\mathbf{E},\operatorname{grad}\operatorname{grad}\tau),& \forall\tau \in \Sigma_{h},\\
(\Pi_{h}\mathbf{E},\mathbf{u}) +(\operatorname{grad} \operatorname{grad} \Pi_{h}\sigma,\mathbf{u})+  (\Pi_{h}\mathbf{B},\operatorname{curl}\mathbf{u})\\ 
\qquad =(\mathbf{E},\mathbf{u}) +(\operatorname{grad} \operatorname{grad} \sigma,\mathbf{u})+  (\mathbf{B},\operatorname{curl}\mathbf{u}),& \forall \mathbf{u} \in \Sigma_{h},\\
(\Pi_{h}\mathbf{B},\mathbf{v})-(\operatorname{curl}\Pi_{h}\mathbf{E},\mathbf{v}) = (\mathbf{B},\mathbf{v}) - (\operatorname{curl}\mathbf{E},\mathbf{v}), &\forall \mathbf{v} \in\widehat{V}_{h}.
\end{cases}
\end{equation}

\subsection{The Solution of the fully discrete system and error estimates}
Suppose that $T = N\Delta t$ with a positive integer $N$.\ Let $u^{j}$ denote the function $u(t_{j})$ with $t_{j} = j\Delta t$ for $j = 0,1,\cdots,N$.\ Define
$$
\partial_{t} u^{j+\frac{1}{2}}=\frac{u^{j+1}-u^{j}}{\Delta t}, \quad \hat{u}^{j+\frac{1}{2}}=\frac{u^{j+1}+u^{j}}{2}.
$$

To discretize the time variable,\ the usual Crank-Nicolson scheme will be used.\ To this end,\ denote by $(\sigma^{j}_{h},\mathbf{E}^{j}_{h},\mathbf{B}^{j}_{h})\in \mathbf{V}_{h}$ the approximation of solution $(\sigma,\mathbf{E},\mathbf{B})$ of (\ref{eq5.1}) at $t_{j}$.\ Given the initial data $(\sigma^{0}_{h},\mathbf{E}^{0}_{h},\mathbf{B}^{0}_{h})\in \mathbf{V}_{h}$,\ for $ 0\leq j \leq N -1$,\ the approximation  $(\sigma^{j+1}_{h},\mathbf{E}^{j+1}_{h},\mathbf{B}^{j+1}_{h})$ at $t_{j+1}$ is defined by
\begin{equation}\label{eq5.4}
\begin{cases} (\partial_{t} \sigma^{j+\frac{1}{2}}_{h},\tau) = (\hat{\mathbf{E}}^{j+\frac{1}{2}}_{h},\operatorname{grad}\operatorname{grad}\tau),& \forall\tau\in U_{h}, \\ (\partial_{t} \mathbf{E}^{j+\frac{1}{2}}_{h},u) =-(\operatorname{grad} \operatorname{grad}\hat{\sigma}^{j+\frac{1}{2}}_{h} ,\mathbf{u})-  (\hat{\mathbf{B}}^{j+\frac{1}{2}}_{h},\operatorname{curl}\mathbf{u}),& \forall \mathbf{u} \in \Sigma_{h}, \\ (\partial_{t} \mathbf{B}^{j+\frac{1}{2}}_{h},\mathbf{v})  =(\operatorname{curl} \hat{\mathbf{E}}^{j+\frac{1}{2}}_{h},\mathbf{v}),& \forall \mathbf{v} \in \widehat{V}_{h}. \end{cases}
\end{equation}
The system of $(\sigma^{j+1}_{h},\mathbf{E}^{j+1}_{h},\mathbf{B}^{j+1}_{h})$ can be written as:
\begin{equation*}
\begin{cases}
(\sigma_{h}^{j+1},\tau) -\frac{\Delta t}{2}(\mathbf{E}_{h}^{j+1},\operatorname{grad}\operatorname{grad}\tau) = (\sigma_{h}^{j},\tau) +\frac{\Delta t}{2}(\mathbf{E}_{h}^{j},\operatorname{grad}\operatorname{grad}\tau),& \forall\tau \in U_{h},\\
(\mathbf{E}_{h}^{j+1},\mathbf{u}) +\frac{\Delta t}{2}(\operatorname{grad} \operatorname{grad} \sigma_{h}^{j+1},\mathbf{u})+  \frac{\Delta t}{2}(\mathbf{B}_{h}^{j+1},\operatorname{curl}\mathbf{u}) =\\ \qquad (\mathbf{E}_{h}^{j},\mathbf{u}) -\frac{\Delta t}{2}(\operatorname{grad} \operatorname{grad} \sigma_{h}^{j},\mathbf{u})-  \frac{\Delta t}{2}(\mathbf{B}_{h}^{j},\operatorname{curl}\mathbf{u}),& \forall \mathbf{u} \in \Sigma_{h},\\
(\mathbf{B}_{h}^{j+1},\mathbf{v})-\frac{\Delta t}{2}(\operatorname{curl}\mathbf{E}_{h}^{j+1},\mathbf{v}) = (\mathbf{B}_{h}^{j},\mathbf{v})+\frac{\Delta t}{2}(\operatorname{curl}\mathbf{E}_{h}^{j},\mathbf{v}), & \forall \mathbf{v} \in \widehat{V}_{h}.
\end{cases}
\end{equation*}
To show the system is nonsingular,\ consider the homogeneous system.\ It will be shown that $\sigma^{j+1}_{h},\mathbf{E}^{j+1}_{h}$ and $\mathbf{B}^{j+1}_{h}$ must vanish.\ Let $\tau = \sigma^{j+1}_{h},\mathbf{u} = \mathbf{E}^{j+1}_{h},\mathbf{v} = \mathbf{B}^{j+1}_{h}$ in the above system and add them together,.\ This leads to $(\sigma^{j+1}_{h},\sigma^{j+1}_{h}) + (\mathbf{E}^{j+1}_{h},\mathbf{E}^{j+1}_{h})+(\mathbf{B}^{j+1}_{h},\mathbf{B}^{j+1}_{h}) =0$.\ It implies that  $\sigma^{j+1}_{h}=0,\mathbf{E}^{j+1}_{h}=0$ and $\mathbf{B}^{j+1}_{h}=0$.

The error estimates are stated in the following theorem.
\begin{theorem}
	Let $(\sigma,\mathbf{E},\mathbf{B})$ and $(\sigma^{j}_{h},\mathbf{E}^{j}_{h},\mathbf{B}^{j}_{h})$ be the solutions of (\ref{eq5.1}) and (\ref{eq5.4}),\ respectively,\ let the initial data $(\sigma^{0}_{h},\mathbf{E}^{0}_{h},\mathbf{B}^{0}_{h}) = \Pi_{h}(\sigma(0),\mathbf{E}(0),\mathbf{B}(0))$.\ Assume that
	\begin{equation*}
	\begin{aligned}
	\sigma& \in W^{1,1}([0,T],H^{k+3}(\Omega;\mathbb{R})\cap C^4(\Omega;\mathbb{R})) \cap W^{3,1}([0,T],L^2(\Omega;\mathbb{R})) \cap L^{\infty}([0,T],H^{k+3}(\Omega;\mathbb{R})),\\
	\mathbf{E}&\in W^{1,1}([0,T],H^{k+1}(\Omega;\mathbb{S})\cap C^2(\Omega;\mathbb{S})) \cap W^{3,1}([0,T],L^2(\Omega;\mathbb{S})) \cap L^{\infty}([0,T],H^{k+1}(\Omega;\mathbb{S})),\\
	\mathbf{B}&\in W^{1,1}([0,T],H^{k}(\Omega;\mathbb{T})) \cap W^{3,1}([0,T],L^2(\Omega;\mathbb{T})) \cap L^{\infty}([0,T],H^{k}(\Omega;\mathbb{T})),
	\end{aligned}
	\end{equation*}
	it holds that,\ for $1 \leq j \leq N$,
	\begin{equation*}
	\begin{aligned}
	\|\sigma^{j} - \sigma_{h}^{j}\|_{0,\Omega} + \|\mathbf{E}^{j} - &\mathbf{E}_{h}^{j}\|_{0,\Omega} + \|\mathbf{B}^{j} - \mathbf{B}_{h}^{j}\|_{0,\Omega} \leq C (h^{k} + \Delta t^2)(\|\sigma\|_{W^{1,1}(H^{k+3}) \cap W^{1,3}(L^2) \cap L^{\infty}(H^{k+3})}\\
	&+\|\mathbf{E}\|_{W^{1,1}(H^{k+1}) \cap W^{1,3}(L^2) \cap L^{\infty}(H^{k+1})} + \|\mathbf{B}\|_{W^{1,1}(H^{k}) \cap W^{1,3}(L^2) \cap L^{\infty}(H^{k})}),
	\end{aligned}
	\end{equation*}
	with constant $C>0$ independent of $h$ and $\Delta t$. 
\end{theorem}
\begin{proof}
	First,\ there exists the following decomposition of the errors:
	$$
	\begin{array}{l}{e_{\sigma}^{j} :=\sigma^{j}_{h}-\sigma^{j}=\left(\sigma^{j}_{h}-\Pi_{h}\sigma^{j}\right)+\left(\Pi_{h}\sigma^{j}-\sigma^{j}\right)=: \theta_{\sigma}^{j}+p_{\sigma}^{j}}, \\ {e_{\mathbf{E}}^{j} :=\mathbf{E}^{j}_{h}-\mathbf{E}^{j}=\left(\mathbf{E}^{j}_{h}-\Pi_{h}\mathbf{E}^{j}\right)+\left(\Pi_{h}\mathbf{E}^{j}-\mathbf{E}^{j}\right)=: \theta_{\mathbf{E}}^{j}+p_{\mathbf{E}}^{j}}, \\ {e_{\mathbf{B}}^{j} :=\mathbf{B}^{j}_{h}-\mathbf{B}^{j}=\left(\mathbf{B}^{j}_{h}-\Pi_{h}\mathbf{B}^{j}\right)+\left(\Pi_{h}\mathbf{B}^{j}-\mathbf{B}^{j}\right)=: \theta_{\mathbf{B}}^{j}+p_{\mathbf{B}}^{j}}.\end{array}
	$$
	The error estimate for the projection errors $(p_{\sigma},p_{\mathbf{E}},p_{\mathbf{B}})$ are already given in \eqref{eq3.1.5}.\ It remains to analyze the priori estimates for the errors  $(\theta_{\sigma}^{j},\theta_{\mathbf{E}}^{j},\theta_{\mathbf{B}}^{j})$.
	
	Setting $t=t_{j}$ and $t=t_{j+1}$ in (\ref{eq5.1}) and taking the arithmetic mean yields 
	\begin{equation}\label{eq5.5}
	\begin{cases} (\hat{\dot{\sigma}}^{j+\frac{1}{2}},\tau)= (\hat{\mathbf{E}}^{j+\frac{1}{2}},\operatorname{grad}\operatorname{grad}\tau),& \forall\tau\in U_{h}, \\ (\hat{\dot{\mathbf{E}}}^{j+\frac{1}{2}},\mathbf{u}) =-(\operatorname{grad} \operatorname{grad}\hat{\sigma}^{j+\frac{1}{2}} ,\mathbf{u})-  (\hat{\mathbf{B}}^{j+\frac{1}{2}},\operatorname{curl}\mathbf{u}),& \forall \mathbf{u} \in \Sigma_{h}, \\ (\hat{\dot{\mathbf{B}}}^{j+\frac{1}{2}},\mathbf{v})  =(\operatorname{curl} \hat{\mathbf{E}}^{j+\frac{1}{2}},\mathbf{v}),& \forall \mathbf{v} \in \widehat{V}_{h}. \end{cases}
	\end{equation}
	Substracting (\ref{eq5.4}) from (\ref{eq5.5}) shows that
	\begin{equation}\label{eq5.6}
	\begin{cases} (\partial_{t}e_{\sigma}^{j+\frac{1}{2}},\tau) +(\partial_{t} \sigma^{j+\frac{1}{2}}-\hat{\dot{\sigma}}^{j+\frac{1}{2}},\tau)= (\hat{e}^{j+\frac{1}{2}}_{\mathbf{E}},\operatorname{grad}\operatorname{grad}\tau),& \forall\tau\in U_{h} ,\\ (\partial_{t}e_{\mathbf{E}}^{j+\frac{1}{2}},\mathbf{u}) +(\partial_{t} \mathbf{E}^{j+\frac{1}{2}}-\hat{\dot{\mathbf{E}}}^{j+\frac{1}{2}},\mathbf{u})=\\ \qquad -(\operatorname{grad} \operatorname{grad}\hat{e}^{j+\frac{1}{2}}_{\sigma} ,\mathbf{u})-  (\hat{e}^{j+\frac{1}{2}}_{\mathbf{B}},\operatorname{curl}\mathbf{u}),&\forall \mathbf{u} \in \Sigma_{h} ,\\ (\partial_{t}e_{\mathbf{B}}^{j+\frac{1}{2}},\mathbf{v}) +(\partial_{t} \mathbf{B}^{j+\frac{1}{2}}-\hat{\dot{\mathbf{B}}}^{j+\frac{1}{2}},\mathbf{v})  =(\operatorname{curl} \hat{e}^{j+\frac{1}{2}}_{\mathbf{E}},\mathbf{v})\,& \forall \mathbf{v} \in \widehat{V}_{h}. \end{cases}
	\end{equation}
	It follows from system (\ref{eq5.3}) that 
	\begin{equation*}
	\begin{cases} (\hat{p}_{\sigma}^{j+\frac{1}{2}},\tau) = (\hat{p}^{j+\frac{1}{2}}_{\mathbf{E}},\operatorname{grad}\operatorname{grad}\tau),& \forall\tau\in U_{h}, \\ (\hat{p}_{\mathbf{E}}^{j+\frac{1}{2}},\mathbf{u}) =-(\operatorname{grad} \operatorname{grad}\hat{p}^{j+\frac{1}{2}}_{\sigma} ,\mathbf{u})-  (\hat{p}^{j+\frac{1}{2}}_{\mathbf{B}},\operatorname{curl}\mathbf{u}),& \forall \mathbf{u} \in \Sigma_{h}, \\ (\hat{p}_{\mathbf{B}}^{j+\frac{1}{2}},\mathbf{v})  =(\operatorname{curl} \hat{p}^{j+\frac{1}{2}}_{\mathbf{E}},\mathbf{v}),& \forall \mathbf{v} \in \widehat{V}_{h}. \end{cases}
	\end{equation*}
	
	Let $\tau = \hat{\theta}_{\sigma}^{j+\frac{1}{2}}, \mathbf{u} = \hat{\theta}_{\mathbf{E}}^{j+\frac{1}{2}}, \mathbf{v} = \hat{\theta}_{\mathbf{B}}^{j+\frac{1}{2}}$ in (\ref{eq5.6}).\ Adding the equations in \eqref{eq5.6} together and using the above system leads to 
	\begin{equation*}
	\begin{array}{l}
	\quad(\|\theta_{\sigma}^{j+1}\|^2_{0,\Omega} + \|\theta_{\mathbf{E}}^{j+1}\|_{0,\Omega}^2 + \|\theta_{\mathbf{B}}^{j+1}\|_{0,\Omega}^2) - (\|\theta_{\sigma}^{j}\|^2_{0,\Omega} + \|\theta_{\mathbf{E}}^{j}\|_{0,\Omega}^2 + \|\theta_{\mathbf{B}}^{j}\|_{0,\Omega}^2)\\
	=2\Delta t(-\partial_{t}p_{\sigma}^{j+\frac{1}{2}} - (\partial_{t} \sigma^{j+\frac{1}{2}}-\hat{\dot{\sigma}}^{j+\frac{1}{2}}) + \hat{p}^{j+\frac{1}{2}}_{\sigma},\hat{\theta}_{\sigma}^{j+\frac{1}{2}})\\
	+2\Delta t(-\partial_{t}p_{\mathbf{E}}^{j+\frac{1}{2}} - (\partial_{t} \mathbf{E}^{j+\frac{1}{2}}-\hat{\dot{\mathbf{E}}}^{j+\frac{1}{2}}) + \hat{p}^{j+\frac{1}{2}}_{\mathbf{E}},\hat{\theta}_{\mathbf{E}}^{j+\frac{1}{2}})\\
	+2\Delta t(-\partial_{t}p_{\mathbf{B}}^{j+\frac{1}{2}} - (\partial_{t} \mathbf{B}^{j+\frac{1}{2}}-\hat{\dot{\mathbf{B}}}^{j+\frac{1}{2}}) + \hat{p}^{j+\frac{1}{2}}_{\mathbf{B}},\hat{\theta}_{\mathbf{B}}^{j+\frac{1}{2}}),
	\end{array}
	\end{equation*}
	An application of the Cauchy-Schwarz inequality proves
	\begin{equation}\label{eq5.7}
	\begin{array}{l}
	\quad(\|\theta_{\sigma}^{j+1}\|^2_{0,\Omega} + \|\theta_{\mathbf{E}}^{j+1}\|_{0,\Omega}^2 + \|\theta_{\mathbf{B}}^{j+1}\|_{0,\Omega}^2)^{\frac{1}{2}} - (\|\theta_{\sigma}^{j}\|^2_{0,\Omega} + \|\theta_{\mathbf{E}}^{j}\|_{0,\Omega}^2 + \|\theta_{\mathbf{B}}^{j}\|_{0,\Omega}^2)^{\frac{1}{2}}\\
	\leq C\Delta t(\|\partial_{t}p_{\sigma}^{j+\frac{1}{2}}\|_{0,\Omega} + \|(\partial_{t} \sigma^{j+\frac{1}{2}}-\hat{\dot{\sigma}}^{j+\frac{1}{2}})\|_{0,\Omega} + \|\hat{p}^{j+\frac{1}{2}}_{\sigma}\|_{0,\Omega}\\
	+\|\partial_{t}p_{\mathbf{E}}^{j+\frac{1}{2}}\|_{0,\Omega} + \|(\partial_{t} \mathbf{E}^{j+\frac{1}{2}}-\hat{\dot{\mathbf{E}}}^{j+\frac{1}{2}}) \|_{0,\Omega}+ \|\hat{p}^{j+\frac{1}{2}}_{\mathbf{E}}\|_{0,\Omega}\\
	+\|\partial_{t}p_{\mathbf{B}}^{j+\frac{1}{2}}\|_{0,\Omega} + \|(\partial_{t} \mathbf{B}^{j+\frac{1}{2}}-\hat{\dot{\mathbf{B}}}^{j+\frac{1}{2}})\|_{0,\Omega} + \|\hat{p}^{j+\frac{1}{2}}_{\mathbf{B}}\|_{0,\Omega}).
	\end{array}
	\end{equation}
	Given $g\in C^{3}[0,T]$,\ the Taylor expansion of $g$ reads
	\begin{equation}\label{eq5.8}
	\Delta t \|\partial_{t}g^{j+\frac{1}{2}}\|_{0,\Omega} = \left\|\int_{t_{j}}^{t_{j+1}} \dot{g} \mathrm{d} s\right\|_{0,\Omega} \leq\int_{t_{j}}^{t_{j+1}} \|\dot{g}\|_{0,\Omega} \mathrm{d} s.
	\end{equation}
	\begin{equation}\label{eq5.9}
	\begin{split}
		\Delta t\|(\partial_{t} g^{j+\frac{1}{2}}-\hat{\dot{g}}^{j+\frac{1}{2}})\|_{0,\Omega} = \frac{1}{2}\left\|2 g^{j+1}-2 g^{j}-\Delta t \dot{g}^{j+1}-\Delta t \dot{g}^{j}\right\|_{0,\Omega}  \\ \leq C \Delta t^{2} \int_{t_{j}}^{t_{j+1}}\|\dddot{g}\|_{0,\Omega} \mathrm{d} s.
	\end{split}
	\end{equation}
	A combination of (\ref{eq5.7}), \eqref{eq5.8} and (\ref{eq5.9}) yields the following estimate:
	\begin{equation}\label{eq5.10}
	\begin{array}{l}
	\quad(\|\theta_{\sigma}^{j+1}\|^2_{0,\Omega} + \|\theta_{\mathbf{E}}^{j+1}\|_{0,\Omega}^2 + \|\theta_{\mathbf{B}}^{j+1}\|_{0,\Omega}^2)^{\frac{1}{2}} - (\|\theta_{\sigma}^{0}\|^2_{0,\Omega} + \|\theta_{\mathbf{E}}^{0}\|_{0,\Omega}^2 + \|\theta_{\mathbf{B}}^{0}\|_{0,\Omega}^2)^{\frac{1}{2}}\\
	\leq C(\int_{0}^{t_{j+1}} \|\dot{p_{\sigma}} \|_{0,\Omega} +\|\dot{p_{\mathbf{E}}} \|_{0,\Omega} + \|\dot{p_{\mathbf{B}}} \|_{0,\Omega} \mathrm{d} s\\
	+\Delta t^{2} \int_{0}^{t_{j+1}} \|\dddot{\sigma} \|_{0,\Omega} +\|\dddot{\mathbf{E}} \|_{0,\Omega} + \|\dddot{\mathbf{B}} \|_{0,\Omega} \mathrm{d} s\\
	+\Delta t\sum_{k=0}^{j+1}(\|p^{k}_{\sigma}\|_{0,\Omega} +\|p^{k}_{\mathbf{E}}\|_{0,\Omega} + \|p^{k}_{\mathbf{B}}\|_{0,\Omega} )).
	\end{array}
	\end{equation}
	Since the initial data $(\sigma^{0}_{h},\mathbf{E}^{0}_{h},\mathbf{B}^{0}_{h}) = \Pi_{h}(\sigma(0),\mathbf{E}(0),\mathbf{B}(0))$,\ it implies that $(\theta_{\sigma}^{0},\theta_{\mathbf{E}}^{0},\theta_{\mathbf{B}}^{0})$ vanish.\ By the estimates of the projection errors in (\ref{eq5.10}), this shows that
	\begin{equation}\label{eq5.11}
	\begin{array}{l}
	\quad (\|\theta_{\sigma}^{j+1}\|^2_{0,\Omega} + \|\theta_{\mathbf{E}}^{j+1}\|_{0,\Omega}^2 + \|\theta_{\mathbf{B}}^{j+1}\|_{0,\Omega}^2)^{\frac{1}{2}} \\
	\leq C(h^{k}\int_{0}^{t_{j+1}} \|\dot{\sigma} \|_{k+3,\Omega} +\|\dot{\mathbf{E}} \|_{k+1,\Omega} + \|\dot{B} \|_{k,\Omega} \mathrm{d} s\\
	+\Delta t^{2} \int_{0}^{t_{j+1}} \|\dddot{\sigma} \|_{0,\Omega} +\|\dddot{\mathbf{E}} \|_{0,\Omega} + \|\dddot{\mathbf{B}} \|_{0,\Omega} \mathrm{d} s\\
	+j\Delta t h^{k}(\|\sigma\|_{L^{\infty}(H^{k+3})} +\|\mathbf{E}\|_{L^{\infty}(H^{k+1})} + \|\mathbf{B}\|_{L^{\infty}(H^{k})})).
	\end{array}
	\end{equation}
	A combination of this and the estimate of the projection errors completes the proof.
\end{proof}
\section{Conclusion}
In this paper,\ the first family of conforming finite elements is constructed for the Gradgrad-complexes in three dimensions,\ and the exactness property is shown for the discrete complexes.\ The complexity and high polynomial degree may limit the practical significance.\ However,\ it provides insights for designing simpler methods.
\bibliographystyle{plain}
\bibliography{reference}

\begin{thebibliography}{10}

\bibitem{MR1634577}
Arlen Anderson, Yvonne Choquet-Bruhat, and James~W. York, Jr.
\newblock Einstein-{B}ianchi hyperbolic system for general relativity.
\newblock {\em Topol. Methods Nonlinear Anal.}, 10(2):353--373, 1997.
\newblock Dedicated to Olga Ladyzhenskaya.

\bibitem{MR1397127}
Helmut Friedrich.
\newblock Hyperbolic reductions for {E}instein's equations.
\newblock {\em Classical Quantum Gravity}, 13(6):1451--1469, 1996.

\bibitem{MR3352360}
Jun Hu.
\newblock Finite element approximations of symmetric tensors on simplicial
  grids in {$\Bbb R^n$}: the higher order case.
\newblock {\em J. Comput. Math.}, 33(3):283--296, 2015.

\bibitem{MR3301063}
Jun Hu and ShangYou Zhang.
\newblock A family of symmetric mixed finite elements for linear elasticity on
  tetrahedral grids.
\newblock {\em Sci. China Math.}, 58(2):297--307, 2015.

\bibitem{MR3529252}
Jun Hu and Shangyou Zhang.
\newblock Finite element approximations of symmetric tensors on simplicial
  grids in {$\Bbb{R}^n$}: the lower order case.
\newblock {\em Math. Models Methods Appl. Sci.}, 26(9):1649--1669, 2016.

\bibitem{MR4113080}
Dirk Pauly and Walter Zulehner.
\newblock The div{D}iv-complex and applications to biharmonic equations.
\newblock {\em Appl. Anal.}, 99(9):1579--1630, 2020.

\bibitem{MR3450102}
Vincent Quenneville-Belair.
\newblock {\em A {N}ew {A}pproach to {F}inite {E}lement {S}imulations of
  {G}eneral {R}elativity}.
\newblock ProQuest LLC, Ann Arbor, MI, 2015.
\newblock Thesis (Ph.D.)--University of Minnesota.

\bibitem{MR1011446}
L.~Ridgway Scott and Shangyou Zhang.
\newblock Finite element interpolation of nonsmooth functions satisfying
  boundary conditions.
\newblock {\em Math. Comp.}, 54(190):483--493, 1990.

\bibitem{MR350260}
Alexander \v{Z}en\'{\i}\v{s}ek.
\newblock Polynomial approximation on tetrahedrons in the finite element
  method.
\newblock {\em J. Approximation Theory}, 7:334--351, 1973.

\bibitem{MR2474112}
Shangyou Zhang.
\newblock A family of 3{D} continuously differentiable finite elements on
  tetrahedral grids.
\newblock {\em Appl. Numer. Math.}, 59(1):219--233, 2009.

\end{thebibliography}

\end{document}